\documentclass[12pt,a4]{amsart}
\usepackage{amsmath}
\usepackage{amssymb}
\usepackage{amsthm}
\usepackage{enumitem}
\usepackage{url}
\usepackage{soul}
\usepackage{times}
\usepackage[T1]{fontenc}

\usepackage{color}
\usepackage{dsfont}
\usepackage{epsfig}
\usepackage[hidelinks]{hyperref}
\usepackage{setspace}
\usepackage{epstopdf}
\usepackage[authoryear,round,sort]{natbib}
\usepackage{comment}
\usepackage{booktabs}
\usepackage{float}
\usepackage[linesnumbered,ruled,vlined]{algorithm2e}
\SetKwInput{KwInput}{Input}                
\SetKwInput{KwOutput}{Output} 
\usepackage{stmaryrd}
\usepackage{subcaption,graphicx}
\usepackage{float}
\numberwithin{equation}{section}

\usepackage[font=small,labelfont=bf]{caption}
\DeclareCaptionFont{tiny}{\tiny}
\usepackage{geometry}
\theoremstyle{plain}
\newtheorem{proposition}{Proposition}[section]

\newtheorem{theorem}{Theorem}[section]
\newtheorem{lemma}{Lemma}[section]
\newtheorem{corollary}{Corollary}[section]

\theoremstyle{definition}
\newtheorem{definition}{Definition}[section]

\theoremstyle{remark}
\newtheorem{rk}{Remark}[section]
\expandafter\let\expandafter\oldproof\csname\string\proof\endcsname
\let\oldendproof\endproof
\renewenvironment{proof}[1][\proofname]{%
  \oldproof[\noindent\textbf{#1.} ]%
}{\oldendproof}

\newcommand{\mP}{\mathbb{P}}
\newcommand{\be}{\begin{equation}}
\newcommand{\ee}{\end{equation}}
\newcommand{\by}{\begin{eqnarray*}}
\newcommand{\ey}{\end{eqnarray*}}

\renewcommand{\leq}{\leqslant}
\renewcommand{\geq}{\geqslant}
\usepackage{xcolor}
\definecolor{dark-red}{rgb}{0.4,0.15,0.15}
\definecolor{dark-blue}{rgb}{0.15,0.15,0.4}
\definecolor{medium-blue}{rgb}{0,0,0.5}
\hypersetup{
    colorlinks, linkcolor={red},
    citecolor={blue}, urlcolor={blue}
}
\allowdisplaybreaks

\begin{document}
\title{An improved variant of simulated annealing that converges under fast cooling}
\author{Michael C.H. Choi}
\address{School of Data Science, The Chinese University of Hong Kong, Shenzhen, Guangdong, 518172, P.R. China and Shenzhen Institute of Artificial Intelligence and Robotics for Society}
\email{michaelchoi@cuhk.edu.cn}
\date{\today}
\maketitle

\begin{abstract}
	Given a target function $U$ to minimize on a finite state space $\mathcal{X}$, a proposal chain with generator $Q$ and a cooling schedule $T(t)$ that depends on time $t$, in this paper we study two types of simulated annealing (SA) algorithms with generators $M_{1,t}(Q,U,T(t))$ and $M_{2,t}(Q,U,T(t))$ respectively. While $M_{1,t}$ is the classical SA algorithm, we introduce a simple and improved variant that we call $M_{2,t}$ which provably converges faster. When $T(t) > c_{M_2}/\log(t+1)$ follows the logarithmic cooling schedule, our proposed algorithm is strongly ergodic both in total variation and in relative entropy, and converges to the set of global minima, where $c_{M_2}$ is a constant that we explicitly identify. If $c_{M_1}$ is the optimal hill-climbing constant that appears in logarithmic cooling of $M_{1,t}$, we show that $c_{M_1} \geq c_{M_2}$ and give simple conditions under which $c_{M_1} > c_{M_2}$. Our proposed $M_{2,t}$ thus converges under a faster logarithmic cooling in this regime. The other situation that we investigate corresponds to $c_{M_1} > c_{M_2} = 0$, where we give a class of fast and non-logarithmic cooling schedule that works for $M_{2,t}$ (but not for $M_{1,t}$). In addition to these asymptotic convergence results, we compare and analyze finite-time behaviour between these two annealing algorithms as well. Finally, we present two algorithms to simulate $M_{2,t}$.
	\smallskip
	
	
\end{abstract}

\tableofcontents

\allowdisplaybreaks

\section{Introduction}

Given a target function $U$ to minimize on a finite state space $\mathcal{X}$, simulated annealing is a popular stochastic optimization algorithm that has found extensive empirical success in diverse disciplines ranging from image processing to statistics and combinatorial optimization problems, see for example \cite{K84, KGv83,BT93,GG84,C04} and the references therein. The algorithm constructs a non-homogeneous continuous-time Markov chain $X^{M_1} = (X^{M_1}_t)_{t \geq 0}$ whose generator depends on the proposal chain as well as the so-called cooling schedule $T(t)$. Roughly speaking, simulated annealing can be considered as a non-homogeneous version of the Metropolis-Hastings algorithm in which the acceptance probability depends on $T(t)$. The cooling schedule $T(t)$ is carefully designed so that in the long run $X^{M_1}$ converges to the set of global minima of $U$. On one hand, $T(t)$ cannot converge too slowly as it is impractical to run $X^{M_1}$ for a long period of time; on the other hand however, $T(t)$ cannot converge too fast as there are well-documented instances in which $X^{M_1}$ may get stuck at a local minimum. \cite{HS88,Gidas85} proved that the optimal cooling schedule for $X^{M_1}$ (in the sense of weak or strong ergodicity) is of the form
$$T(t) = \dfrac{c_{M_1}}{\log(t+1)},$$
where $c_{M_1}$ is the hill-climbing constant to be introduced in \eqref{eq:cm1} below. In practice however, people adapt fast cooling schedule even though they usually do not come along with convergence guarantee.

In this paper, inspired by the recent work of the author \cite{Choi16, CH18} who studied a new variant of the Metropolis-Hastings algorithm, we propose a promising variant of simulated annealing $X^{M_2} = (X^{M_2}_t)_{t \geq 0}$ that enjoys superior mixing properties, provably converges under fast cooling and in some cases does not suffer from the drawbacks of $X^{M_1}$ as mentioned above. Precisely, our contributions are the following:

\begin{enumerate}
	\item \textbf{Derive basic yet important properties of $X^{M_2}$}: We prove a few elementary properties and compare them with their $X^{M_1}$ counterparts in Lemma \ref{lem:M1M2}. These simple results turn out to be crucial in proving our main results. In particular, the difference in the convergence behaviour between $X^{M_1}$ and $X^{M_2}$ stems from the difference between their quadratic forms. The spectral gap lower bound of $X^{M_2}$ presented in Lemma \ref{lem:spectralgaplower} is also of independent interest.
	
	
	\item \textbf{Strong ergodicity of $X^{M_2}$ under fast cooling}: We split our ergodicity results into two regime according to $c_{M_2} > 0$ or $c_{M_2} \leq 0$, where $c_{M_2}$ is a constant that we explicitly identify in \eqref{eq:cm2} below and depends on $U$ and $Q$. In the first case when $c_{M_2} > 0$ in Theorem \ref{thm:stronge}, we establish rigorous convergence guarantee (in total variation and in relative entropy) of $X^{M_2}$ when the cooling schedule $T(t)$ is of the form
	$$T(t) = \dfrac{c_{M_2} + \epsilon}{\log(t+1)},$$
	for all $\epsilon > 0$. Thus the speed-up of the proposed variant depends on the difference $c_{M_1} - c_{M_2}$.
	
	In the second situation when $c_{M_2} \leq 0$ in Theorem \ref{thm:stronge2} and Theorem \ref{thm:relativee2}, we give a class of cooling schedule that is faster than logarithmic cooling and prove the strong ergodicity of $X^{M_2}$ in such setting. Note that in a continuous time and state space setting, \cite{FQG97} proposes an improved variant of the Langevin diffusion, where the cooling schedule is accelerated in a similar spirit as in this paper with a smaller hill-climbing parameter.
	
	
	\item \textbf{Finite-time convergence estimate of $X^{M_2}$}: In addition to asymptotic convergence theorems, we prove a finite-time convergence estimate of $X^{M_2}$ in Corollary \ref{thm:finite}, as a simple application of the general result in \cite{DM94} for non-homogeneous Markov chains.
	
	\item \textbf{$X^{M_2}$ effectively escapes local minima under general cooling}: It is well-known that $X^{M_1}$ may get trapped at a local minimum when one adapts fast cooling. As a result, one may have similar concern for $X^{M_2}$. In Theorem \ref{thm:escape}, we prove that such concern is not valid and $X^{M_2}$ effectively escapes local minima with probability tends to $1$ as time goes to infinity under general cooling. In other words, $X^{M_2}$ can still effectively explore the state space even under poor initialization and fast cooling.
	
	\item \textbf{Propose two algorithms to simulate $X^{M_2}$}: While $X^{M_1}$ can be efficiently simulated using the idea of acceptance-rejection, it seems difficult to adapt such procedure to our proposed variant $X^{M_2}$. We propose two algorithms to simulate $X^{M_2}$ in Algorithm \ref{algo:M21} and Algorithm \ref{algo:M2}.
	
	\item \textbf{Discuss possible limitations in deploying $X^{M_2}$}: While $X^{M_2}$ enjoys some desirable theoretical properties, it may take extra computational costs to deploy it in practice when compared with classical $X^{M_1}$.
\end{enumerate}

The rest of this paper is organized as follow. In Section \ref{sec:prelim}, we fix our notations and introduce the classical simulated annealing $X^{M_1}$ as well as the proposed variant $X^{M_2}$. This is then followed by a quick review on various notions of ergodicity for non-homogeneous Markov chains. We proceed to discuss the main results in Section \ref{sec:main}, and their proofs can be found in Section \ref{sec:proof}. We present two algorithms for our proposed variant in Section \ref{sec:algo}, and discuss possible drawbacks of $X^{M_2}$ in the Epilogue in Section \ref{sec:epilogue}. Finally, we conclude our paper with a list of possible future research directions in Section \ref{sec:conclusion}.

\section{Preliminaries}\label{sec:prelim}

\subsection{Two types of simulated annealing: $X^{M_1}$ and its improved variant $X^{M_2}$}\label{subsec:M1M2}

In this section, we introduce the two types of simulated annealing that will be the focus of this paper. Along the way we will fix a few notations. Suppose that we are given the task to minimize a function $U: \mathcal{X} \to \mathbb{R}$ living on a finite state space $\mathcal{X}$. In the setting of simulated annealing, there is an ergodic homogeneous and continuous-time Markov chain with generator $Q$ and stationary distribution $\pi$ that acts as the proposal chain. We assume that $Q$ is reversible with respect to $\pi$, that is, the detailed balance condition is satisfied with $\pi(x) Q(x,y) = \pi(y) Q(y,x)$ for all $x,y \in \mathcal{X}$. Equvialently, $Q$ is a self-adjoint operator in the Hilbert space $\ell^2(\pi)$, endowed with the usual inner product
$$ \langle f,g \rangle_{\pi} := \sum_{x \in \mathcal{X}} f(x) g(x) \pi(x), \quad f,g: \mathcal{X} \rightarrow \mathbb{R}.$$
For any reversible $Q$, it is well-known that the quadratic form of $-Q$ can be written as
\begin{align}\label{eq:quadraticform}
	\langle -Qf,f \rangle_{\pi} = \dfrac{1}{2} \sum_{x,y \in \mathcal{X}} (f(y)-f(x))^2 Q(x,y)\pi(x).
\end{align}
For any reversible ergodic generator $-Q$, we arrange its eigenvalues in non-decreasing order and write $0 = \lambda_1(-Q) < \lambda_2(-Q) \leq \lambda_3(-Q) \ldots \leq \lambda_{|\mathcal{X}|}(-Q)$. It is well-known that $\lambda_2(-Q)$ is the spectral gap of $Q$ and admits a variational formula given by
\begin{align}\label{eq:spectralgap}
\lambda_2(-Q) = \inf_{f \in \ell^2(\pi): \pi(f) = 0} \dfrac{\langle -Qf,f \rangle_{\pi}}{\langle f,f \rangle_{\pi}}.
\end{align} 
Apart from the proposal chain $Q$, another critical component in simulated annealing is a monotonically decreasing function $T(t)$ that we call the temperature or the cooling schedule. We assume that $T(t) > 0$ and decreases to $0$ as $t \to \infty$. We write $\pi_{T(t)}$ to be the Gibbs distribution with probability mass function given by
$$\pi_{T(t)}(x) = \dfrac{e^{-\frac{U(x)}{T(t)}}\pi(x)}{Z_{T(t)}}, \quad x \in \mathcal{X},$$
where $Z_{T(t)} = \sum_{x \in \mathcal{X}} e^{-\frac{U(x)}{T(t)}}\pi(x)$ is the normalization constant. For any $x,y \in \mathbb{R}$, we also denote $x \vee y := \max\{x,y\}, x \wedge y := \min\{x,y\}$ and $x_+ := x \vee 0$. At each time $t$, classical simulated annealing amounts to a Metropolis-Hastings or acceptance-rejection procedure: a move is proposed by the proposal chain with generator $Q$, and is accepted with probability that depends on $T(t)$ and $U$ in a way such that the Markov chain generated by the algorithm at time $t$ is reversible with respect to $\pi_{T(t)}$. Precisely, we have  

\begin{definition}[Classical simulated annealing $X^{M_1}$]\label{def:M1}
Given a target function $U$ on finite state space $\mathcal{X}$, a proposal continuous-time ergodic Markov chain with generator $Q$ and a cooling schedule $T(t)$, the simulated annealing algorithm $X^{M_1} = (X^{M_1}_t)_{t \geq 0}$ is a non-homogeneous Markov chain with generator given by $M_{1,t} = M_{1,t}(Q,U,T(t)) = (M_{1,t}(x,y))_{x,y \in \mathcal{X}}$ for $t \geq 0$, where
$$M_{1,t}(x,y) := \begin{cases} Q(x,y) \min \left\{ 1,e^{\frac{U(x)-U(y)}{T(t)}} \right\} = Q(x,y) e^{-\frac{(U(y)-U(x))_+}{T(t)}}, &\mbox{if } x \neq y; \\
- \sum_{z: z \neq x} M_{1,t}(x,z), & \mbox{if } x = y. \end{cases}$$
We write $P^{M_1} = (P^{M_1}_{s,t})_{0 \leq s \leq t} = (P^{M_1}_{s,t}(x,y))_{x,y \in \mathcal{X},s \leq t}$ to be the transition semigroup of $X^{M_1}$, where $P^{M_1}_{s,t}(x,y)$ is the transition probability of $X^{M_1}$ starting in state $x$ at time $s$ to state $y$ at time $t$.
\end{definition}

Recent advances in simulated annealing includes investigating piecewise deterministic Markov processes and their annealing variants, see for example \cite{Mon16}. Inspired by the recent work by the author \cite{Choi16,CH18}, we would like to introduce a variant of $X^{M_1}$ that we call $X^{M_2}$. It can be constructed by mirroring the transition effect of $X^{M_1}$ to capture the opposite movement. More precisely, we have

\begin{definition}[The improved variant $X^{M_2}$]\label{def:M2}
Given a target function $U$ on finite state space $\mathcal{X}$, a proposal continuous-time ergodic Markov chain with generator $Q$ and a cooling schedule $T(t)$, the improved variant $X^{M_2} = (X^{M_2}_t)_{t \geq 0}$ is a non-homogeneous Markov chain with generator given by $M_{2,t} = M_{2,t}(Q,U,T(t)) = (M_{2,t}(x,y))_{x,y \in \mathcal{X}}$ for $t \geq 0$, where
$$M_{2,t}(x,y) := \begin{cases} Q(x,y) \max \left\{ 1,e^{\frac{U(x)-U(y)}{T(t)}} \right\} = Q(x,y) e^{\frac{(U(x)-U(y))_+}{T(t)}} = e^{\frac{|U(x)-U(y)|}{T(t)}} M_{1,t}(x,y), &\mbox{if } x \neq y; \\
- \sum_{z: z \neq x} M_{2,t}(x,z), & \mbox{if } x = y. \end{cases}$$
We write $P^{M_2} = (P^{M_2}_{s,t})_{0 \leq s \leq t} = (P^{M_2}_{s,t}(x,y))_{x,y \in \mathcal{X},s \leq t}$ to be the transition semigroup of $X^{M_2}$, where $P^{M_2}_{s,t}(x,y)$ is the transition probability of $X^{M_2}$ starting in state $x$ at time $s$ to state $y$ at time $t$.
\end{definition}

\begin{rk}[On Freidlin-Wentzell theory]
	It is well-known that the classical simulated annealing chain $X^{M_1}$ fits perfectly into the framework of Freidlin-Wentzell theory \cite{FW12}, and this perspective has been employed by many researchers to derive fine eigenvalues and convergence estimates of $X^{M_1}$, see for example \cite{CC88, DM94,DelMiclo99,Miclo96}. However, $X^{M_2}$ does not fit into this Freidlin-Wentzell framework since $M_{2,t}$ does not have exponentially vanishing transition rate. 
\end{rk}

Comparing Definition \ref{def:M1} and \ref{def:M2}, we say that $X^{M_2}$ is a improved variant of $X^{M_1}$ in the following sense: starting at a state $x$ the transition rate of $X^{M_2}$ to any other state $y \neq x$ is greater than that of $X^{M_1}$ at any time $t$. Mathematically we see that in defining $M_1$ we take $\min$ while in $M_2$ we consider $\max$ for off-diagonal entries. Intuitively, we can imagine $Q$ as the base transition rate. $X^{M_1}$ and $X^{M_2}$ both modify this base rate $Q$ according to their own (and opposite) rules. If $U(y) \leq U(x)$, $X^{M_1}$  leaves $Q$ unchanged while $X^{M_2}$  increases this rate to $Q(x,y)e^{\frac{U(x)-U(y)}{T(t)}}$. The larger the difference between $U(x)$ and $U(y)$, the greater the ``boost'' on the base transition rate for $X^{M_2}$. On the other hand if $U(y) > U(x)$, $X^{M_1}$ lower this base rate to $Q(x,y)e^{-\frac{(U(y)-U(x))}{T(t)}}$ while $X^{M_2}$ leaves the base rate $Q$ unchanged. We will see that these two key differences allow $X^{M_2}$ to converge under fast cooling schedule and is able to escape local minima effectively.

The main goal of the paper is to derive convergence theorems for the improved variant $X^{M_2}$ under perhaps faster cooling schedule $T(t)$, and to compare the properties and behaviour between $X^{M_1}$ and $X^{M_2}$. To allow for effective comparison between these generators, we recall the notion of Peskun ordering of continuous-time Markov chains. This partial ordering was first introduced by \cite{Pesk73} for discrete time Markov chains on finite state space. It was further generalized by \cite{Tie98} and \cite{LM08} to general state space and to continuous-time Markov chains respectively.

\begin{definition}[Peskun ordering]
	Suppose that there are two continuous-time Markov chains with generators $Q_1$ and $Q_2$ respectively, and both chains share the same stationary distribution $\pi$. $Q_1$ is said to dominate $Q_2$ off-diagonally, written as $Q_1 \succeq Q_2$, if for all $x \neq y \in \mathcal{X}$, we have
	$$Q_1(x,y) \geq Q_2(x,y).$$
\end{definition}

In the following, we collect a few elementary yet important observations and results on the behaviour of generators $M_{1,t}$ and $M_{2,t}$ at a fixed time $t \geq 0$. These results will be repeatedly used to develop our main results in Section \ref{sec:main}.

\begin{lemma}\label{lem:M1M2}
	Given a target function $U$ on finite state space $\mathcal{X}$, a proposal continuous-time ergodic Markov chain with generator $Q$ and a cooling schedule $T(t)$, at a fixed $t \geq 0$ we have
	\begin{enumerate}
		\item(Reversibility) $M_{1,t}$ and $M_{2,t}$ are reversible with respect to the Gibbs distribution $\pi_{T(t)}$. \label{it:1}
		\item(Peskun ordering) $M_{2,t} \succeq M_{1,t}$. \label{it:2}
		\item(Quadratic form) \label{it:3} For any function $f:\ \mathcal{X}\rightarrow \mathbb{R}$ with $f \in \ell^2(\pi_{T(t)})$,
		$$\langle M_{2,t} f,f \rangle_{\pi_{T(t)}} \leq \langle M_{1,t} f,f \rangle_{\pi_{T(t)}},$$
		where
		\begin{align*}
			\langle -M_{2,t} f,f \rangle_{\pi_{T(t)}} &= \dfrac{1}{2 Z_{T(t)}} \sum_{x,y \in \mathcal{X}} (f(y)-f(x))^2 e^{-\frac{\min\{U(x),U(y)\}}{T(t)}}Q(x,y)\pi(x), \\
			\langle -M_{1,t} f,f \rangle_{\pi_{T(t)}} &= \dfrac{1}{2 Z_{T(t)}} \sum_{x,y \in \mathcal{X}} (f(y)-f(x))^2 e^{-\frac{\max\{U(x),U(y)\}}{T(t)}}Q(x,y)\pi(x).
		\end{align*}
		\item(Spectral gap) \label{it:4}
		$$\lambda_2(-M_{2,t}) \geq \lambda_2(-M_{1,t}).$$
	\end{enumerate}
\end{lemma}

\begin{proof}

We first prove item \eqref{it:1}. For any $x \neq y$, we have
$$
\pi_{T(t)}(x)M_{2,t}(x,y) = \dfrac{e^{-\frac{U(x)}{T(t)}}\pi(x)}{Z_{T(t)}}Q(x,y) e^{\frac{(U(x)-U(y))_+}{T(t)}} = \dfrac{e^{-\frac{U(y)}{T(t)}}\pi(y)}{Z_{T(t)}}Q(y,x) e^{\frac{(U(y)-U(x))_+}{T(t)}} = \pi_{T(t)}(y)M_{2,t}(y,x),
$$
where in the second equality we use the reversibility of $Q$ with respect to $\pi$ and $\min\{U(x),U(y)\} = U(x) - (U(x) - U(y))_+ = U(y) - (U(y)-U(x))_+$. $M_2$ is therefore reversible with respect to $\pi_{T(t)}$. Similarly, the reversibility of $M_{1,t}$ can be deduced via 
$$
\pi_{T(t)}(x)M_{1,t}(x,y) = \dfrac{e^{-\frac{U(x)}{T(t)}}\pi(x)}{Z_{T(t)}}Q(x,y) e^{\frac{-(U(y)-U(x))_+}{T(t)}} = \dfrac{e^{-\frac{U(y)}{T(t)}}\pi(y)}{Z_{T(t)}}Q(y,x) e^{\frac{-(U(x)-U(y))_+}{T(t)}} = \pi_{T(t)}(y)M_{1,t}(y,x),
$$
where we use again the reversibility of $Q$ and $\max\{U(x),U(y)\} = U(x) + (U(y)-U(x))_+ = U(y) + (U(x) - U(y))_+.$

Next, we prove item \eqref{it:2}. For the off-digaonal entries with $x \neq y$,
	$$
M_{2,t}(x,y) = Q(x,y) \max \left\{ 1,e^{\frac{U(x)-U(y)}{T(t)}}\right\} \geq   Q(x,y) \min \left\{ 1,e^{\frac{U(x)-U(y)}{T(t)}}\right\} = M_{1,t}(x,y).
$$
Using both item \eqref{it:1} and item \eqref{it:2}, item \eqref{it:3} follows readily from \cite[Theorem $5$]{LM08} since $M_{2,t} \succeq M_{1,t}$ and they are both reversible with respect to $\pi_{T(t)}$. In addition, using \eqref{eq:quadraticform} and 
\begin{align*}
\pi_{T(t)}(x)M_{2,t}(x,y) &= \dfrac{1}{ Z_{T(t)}} e^{-\frac{\min\{U(x),U(y)\}}{T(t)}}Q(x,y)\pi(x), \\
\pi_{T(t)}(x)M_{1,t}(x,y) &= \dfrac{1}{ Z_{T(t)}} e^{-\frac{\max\{U(x),U(y)\}}{T(t)}}Q(x,y)\pi(x),
\end{align*}
we have
\begin{align*}
\langle -M_{2,t}f,f \rangle_{\pi_{T(t)}} &= \dfrac{1}{2} \sum_{x,y \in \mathcal{X}} (f(y)-f(x))^2 M_{2,t}(x,y)\pi_{T(t)}(x) \\
&= \dfrac{1}{2 Z_{T(t)}} \sum_{x,y \in \mathcal{X}} (f(y)-f(x))^2 e^{-\frac{\min\{U(x),U(y)\}}{T(t)}}Q(x,y)\pi(x).
\end{align*}
Similar expression can be obtained for $M_{1,t}$. Finally, we prove item \eqref{it:4}. It follows easily from item \eqref{it:3} and the variational principle for spectral gap \eqref{eq:spectralgap}.
\end{proof}

\subsection{Review on ergodicity of non-homogeneous Markov chains}

In this section, we will give a short detour on various notions of ergodicity of non-homogeneous Markov chains. As the classical simulated annealing $X^{M_1}$ and its improved variant $X^{M_2}$ are non-homogeneous Markov chains, these notions of ergodicity are particularly important in order to properly understand our main results in Section \ref{sec:main}. In particular, we will apply Lemma \ref{lem:stronge} and Lemma \ref{lem:relativee} below to derive our main results. To keep our notations consistent with the previous section, we assume a non-homogeneous Markov chain with generator $Q_t$ that depends on time $t$ and Markov semigroup $(P_{s,t})_{0 \leq s \leq t}$ on state space $\mathcal{X}$. For any two probability measures $\mu$ and $\nu$ on $\mathcal{X}$, we write 
$$||\mu - \nu||_{TV} := \dfrac{1}{2} \sum_{x \in \mathcal{X}} |\mu(x) - \nu(x)|$$
to be the total variation distance between $\mu$ and $\nu$. We now recall the notions of strong ergodicity and weak ergodicity of non-homogeneous Markov chains.

\begin{definition}[Strong ergodicity and weak ergodicity]\label{def:weakestronge}
	Let $X = (X_t)_{t \geq 0}$ be a non-homogeneous continuous-time Markov chain with generator $Q_t$ and Markov semigroup $(P_{s,t})_{0 \leq s \leq t}$ on a finite state space $\mathcal{X}$. $X$ is said to be strongly ergodic if there exists a probability measure $\mu$ on $\mathcal{X}$ such that for all $s \geq 0$, we have
	$$\lim_{t \to \infty} \sup_{x \in \mathcal{X}} ||P_{s,t}(x,\cdot) - \mu||_{TV} = 0.$$
	$X$ is said to be weakly ergodic if for all $s \geq 0$, we have
	$$\lim_{t \to \infty} \sup_{x,y \in \mathcal{X}} ||P_{s,t}(x,\cdot) - P_{s,t}(y,\cdot)||_{TV} = 0.$$
\end{definition}

It is easy to see that strong ergodicity implies weak ergodicity using the triangle inequality of total variation distance. In the setting of classical simulated annealing $X^{M_1}$, necessary and sufficient conditions for weak ergodicity are known, see e.g. \cite{Miclo95}. 

%
%

Sufficient condition for strong ergodicity in terms of the transition rates $Q_t$ can be found in \cite{JI88}. We now give a sufficient condition for strong ergodicity in terms of the spectral gap of $Q_t$:

\begin{lemma}[Sufficient condition for strong ergodicity \cite{Gidas85}]\label{lem:stronge}
	Let $X = (X_t)_{t \geq 0}$ be a non-homogeneous continuous-time Markov chain with generator $Q_t$ and Markov semigroup $(P_{s,t})_{0 \leq s \leq t}$ on a finite state space $\mathcal{X}$. Assume that $Q_t$ is reversible with respect to a probability measure $\mu_t$ and its spectral gap is $\lambda_2(-Q_t)$. If there exists a function $\gamma(t)$ and a probability measure $\mu$ such that
	\begin{align*}
	\left| \dfrac{d \mu_t(x)}{dt} \right| &\leq \gamma(t)\mu_t(x), \quad x \in \mathcal{X},\\
	\int_0^{\infty} \lambda_2(-Q_t)\,dt &= \infty, \\
	\lim_{t \to \infty} \dfrac{\gamma(t)}{\lambda_2(-Q_t)} &= 0, \\
	\lim_{t \to \infty} ||\mu_t - \mu||_{TV} &= 0,
	\end{align*}
	then $X$ is strongly ergodic and converges to $\mu$ in total variation distance.
\end{lemma}

Another notion of ergodicity that we will study is convergence in relative entropy. We write $\mathrm{Ent}_{\nu}(\mu)
$ to denote the relative entropy of $\mu$ with respect to $\nu$, that is,
$$\mathrm{Ent}_{\nu}(\mu) := \sum_{x \in \mathcal{X}} \mu(x) \log \left( \dfrac{\mu(x)}{\nu(x)}\right).$$

\begin{definition}[Convergence in relative entropy]\label{def:relativee}
	Let $X = (X_t)_{t \geq 0}$ be a non-homogeneous continuous-time Markov chain with generator $Q_t$ and Markov semigroup $(P_{s,t})_{0 \leq s \leq t}$ on a finite state space $\mathcal{X}$. Suppose further that the stationary measure of $Q_t$ is $\mu_t$ at a fixed $t \geq 0$. We say that $X$ converges in relative entropy if for every $x \in \mathcal{X}$,
	$$\lim_{t \to \infty} \mathrm{Ent}_{\mu_t}(P_{0,t}(x,\cdot)) = 0.$$
\end{definition}

Note that convergence in relative entropy of the classical simulated annealing $X^{M_1}$ are discussed in \cite{Miclo92,DelMiclo99}. We now state a sufficient condition for convergence in relative entropy in the setting of $X^{M_1}$ and $X^{M_2}$:

\begin{lemma}[Sufficient condition for convergence in relative entropy \cite{Miclo92,DelMiclo99}]\label{lem:relativee}
	For $i = 1,2$, let $X^{M_i} = (X^{M_i}_t)_{t \geq 0}$ be the non-homogeneous continuous-time Markov chain introduced in Definition \ref{def:M1} and \ref{def:M2} with generator $M_{i,t}$ and Markov semigroup $(P_{s,t}^{M_i})_{0 \leq s \leq t}$ on a finite state space $\mathcal{X}$, where $M_{i,t}$ are both reversible with respect to $\pi_{T(t)}$. If there exists a constant $R \geq 0$, a function $a_{T(t)}$ and the cooling schedule is selected such that for $x \in \mathcal{X}$,
	\begin{align*}
		\dfrac{d}{dt} \mathrm{Ent}_{\pi_{T(t)}} (P_{0,t}^{M_i}(x,\cdot)) &\leq -a_{T(t)}^{-1}\mathrm{Ent}_{\pi_{T(t)}} (P_{0,t}^{M_i}(x,\cdot)) + R \left|\dfrac{d}{dt}T(t)\right| \dfrac{1}{T(t)^2}, \\
		\int_0^{\infty} a_{T(t)}^{-1} \,dt &= \infty, \\
		\lim_{t \to \infty} \left|\dfrac{d}{dt}T(t)\right| \dfrac{a_{T(t)}}{T(t)^2} &= 0,
	\end{align*}
	then $X^{M_i}$ converges in relative entropy, i.e. for $x \in \mathcal{X}$,
	$$\lim_{t \to \infty} \mathrm{Ent}_{\pi_{T(t)}} (P_{0,t}^{M_i}(x,\cdot)) = 0.$$
\end{lemma}

We remark that the proof of Lemma \ref{lem:stronge} and Lemma \ref{lem:relativee} are very similar: the conditions in these lemmas guarantee that, for some functions $g,h$ and differentiable function $f$ that satisfy
\begin{align*}
	\dfrac{d}{dt}f(t) &\leq - g(t) f(t) + h(t), \quad
	\int_0^{\infty} g(t) \,dt = \infty,	\quad
	\lim_{t \to \infty} \dfrac{h(t)}{g(t)} = 0,
\end{align*}
then $\lim_{t \to \infty} f(t) = 0$. Another related remark is that log-Sobolev inequalities are used in \cite{Miclo92} to establish the result in Lemma \ref{lem:relativee}.

\section{Main results}\label{sec:main}

This section contains the main results of this paper, which can be classified into the following three categories:
\begin{itemize}
	
	\item In Theorem \ref{thm:stronge}, we discuss strong ergodicity and convergence in relative entropy of $X^{M_2}$ under logarithmic cooling. In some cases depending on $U$ and $Q$, we provide convergence guarantee under faster than logarithmic cooling in Theorem \ref{thm:stronge2} and Theorem \ref{thm:relativee2}.
	
	\item In additional to the asymptotic convergence results for $X^{M_2}$, we provide finite-time convergence estimates for $X^{M_2}$ in Corollary \ref{thm:finite}.
	
	\item In Theorem \ref{thm:escape}, we show that $X^{M_2}$ escapes from local minimum under general cooling.
\end{itemize} 

The proof of these results will be presented in Section \ref{sec:proof}.

Before we proceed to discuss strong ergodicity of $X^{M_2}$, we introduce a few parameters that will play a fundamental role in the cooling schedule. Similar to \cite{HS88,Lowe96,I94}, we say that a path from $x$ to $y$ is any sequence of points starting from $x_0 = x, x_1, x_2,\ldots, x_n =y$ such that $Q(x_{i-1},x_i) > 0$ for $i = 1,2,\ldots,n$. Irreducibility of $Q$ guarantees such path exists for any $x \neq y$. Let $\Gamma^{x,y}$ denote the set of paths from $x$ to $y$, and elements of $\Gamma^{x,y}$ are denoted by $\gamma = (\gamma_i)_{i=0}^n$. If the value of $U(x)$ is considered as the elevation at $x$, then the highest elevation along a path $\gamma \in \Gamma^{x,y}$ is $\mathrm{Elev}_1(\gamma) = \max\{U(\gamma_i);~\gamma_i \in \gamma\}$, and the lowest possible highest elevation along any path from $x$ to $y$ is 
$$H_1(x,y) := \min\{\mathrm{Elev}_1(\gamma);~\gamma \in \Gamma^{x,y}\}.$$
As we shall see later in the main results, it turns out that the natural counterpart of $\mathrm{Elev}_1(\gamma)$ and $H_1(x,y)$ are $\mathrm{Elev}_2(\gamma)$ and $H_2(x,y)$ respectively: 
\begin{align*}
	\mathrm{Elev}_2(\gamma) &:= \max\{U(\gamma_{i-1}) \wedge U(\gamma_{i});~\gamma_i \in \gamma, i = 1,2,\ldots,n\}, \\
	H_2(x,y) &:= \min\{\mathrm{Elev}_2(\gamma);~\gamma \in \Gamma^{x,y}\}.
\end{align*}
We also write
\begin{align}
	U_{min} &:= \{x \in \mathcal{X};~ U(x) = \min_y U(y) \}, \nonumber \\
	\pi_{min}(x) &:= \begin{cases}
	\dfrac{\pi(x)}{\pi(U_{min})}, \quad &\mathrm{if}~ x \in U_{min}, \\
	0, \quad &\mathrm{if}~ x \notin U_{min}.
	\end{cases} \nonumber \\
	R &:= \max_x U(x) - \min_x U(x), \label{eq:R}\\
	K &:= \pi(U_{min})^{-1} - 1, \quad B := \min_{x \in U_{min}^c} U(x) - \min_x U(x). \label{eq:AB}
\end{align}
Two states $x$ and $y$ are said to be equivalent if there exists a path $\gamma^{x,y} = (\gamma_i^{x,y})_{i=0}^n$ of constant energy from $x$ to $y$, that is, $U(\gamma_i^{x,y}) = U(\gamma_{i+1}^{x,y})$ for $i = 0,\ldots,n-1$. A point $x$ is said to be a local minimum either if the set $\{y \in \mathcal{X};~U(y) < U(x)\}$ is empty (in such a case $x$ is a point of global minimum) or for each $y$ with $U(y) < U(x)$ and for each path $\gamma^{x,y} = \Gamma^{x,y}$, there exists $z \in \gamma^{x,y}$ such that $U(z) > U(x)$.

With these notations in mind, the two hill-climbing parameters that we are interested in are
\begin{align}
	c_{M_1} = c_{M_1}(Q,U) &:= \max_{x,y \in \mathcal{X}}\{H_1(x,y) - U(x) - U(y)\}, \label{eq:cm1} \\
	c_{M_2} = c_{M_2}(Q,U) &:= \max_{x,y \in \mathcal{X}}\{H_2(x,y) - U(x) - U(y)\}. \label{eq:cm2}
\end{align}

In our subsequent analysis, we shall assume without loss of generality that 
$$\min_{x \in \mathcal{X}} U(x) = 0.$$
Note that the same assumption is also imposed in the simulated annealing literature such as \cite{HS88}.

We now state a few basic properties regarding these two constants:
\begin{proposition}[Comparison and characterizations of $c_{M_1}$ and $c_{M_2}$]\label{prop:comparecm1cm2}
	Suppose that $c_{M_i}$ are defined as in \eqref{eq:cm1} and \eqref{eq:cm2} for $i = 1,2$. We have
	\begin{enumerate}
		\item\label{it:cm1cm2} $c_{M_1} \geq c_{M_2}$. In particular, when $U$ does not have repeated values, $c_{M_1} > c_{M_2}$.
		\item\label{it:cm1nonnegative} $c_{M_1} \geq 0$. 
		\item\label{it:cm2negative} ($c_{M_2}$ can be negative) For $x \neq y \in \mathcal{X}$, we write $s^{x,y}$ to be the second largest element along the path $\gamma^{x,y}$ that achieves the lowest elevation $H_1(x,y)$, that is,
		\begin{align*}
			U^{x,y}_{max} &:= \{w \in \gamma^{x,y};~ U(w) = \mathrm{Elev}_1(\gamma^{x,y}) = H_1(x,y)\} \\
			s^{x,y} &:= \begin{cases}
			\max\{U(w);~ w \in \gamma^{x,y} \backslash U^{x,y}_{max}\}, \quad &\mathrm{if}~|U^{x,y}_{max}| = 1, \\
			\mathrm{Elev}_1(\gamma^{x,y}) = H_1(x,y), \quad &\mathrm{if}~|U^{x,y}_{max}| > 1.
			\end{cases}
		\end{align*}
		If $s^{x,y} \leq U(x) + U(y)$, then $c_{M_2} \leq 0$.
		\item\label{it:whencm2negative} (Characterization of negative $c_{M_2}$) Assume that $U$ does not have repeated values, then $c_{M_2} < 0$ if and only if $c_{M_1} = 0$ if and only if $U$ has $1$ local minimum (up to equivalence).
	\end{enumerate}
\end{proposition}

Intuitively speaking, $c_{M_1}$ is the minimal elevation to climb from a local minimum to a global minimum, while $c_{M_2}$ is bounded above by the second largest hill to climb along any path from $x$ to $y$. To gain a better understanding of these two parameters, we plot four graphs in Figure \ref{fig:cM1cM2} as a simple illustration. In Figure \ref{fig:cM1cM2}, $U$ is a one-dimensional function consisting only of crosses on the graph, and we take the proposal chain $Q$ to be a birth-death process that explores neighbouring points on the left or on the right.

\begin{figure}[h] 
	\begin{subfigure}[b]{0.5\linewidth}
		\centering
		\includegraphics[width=1.1\linewidth]{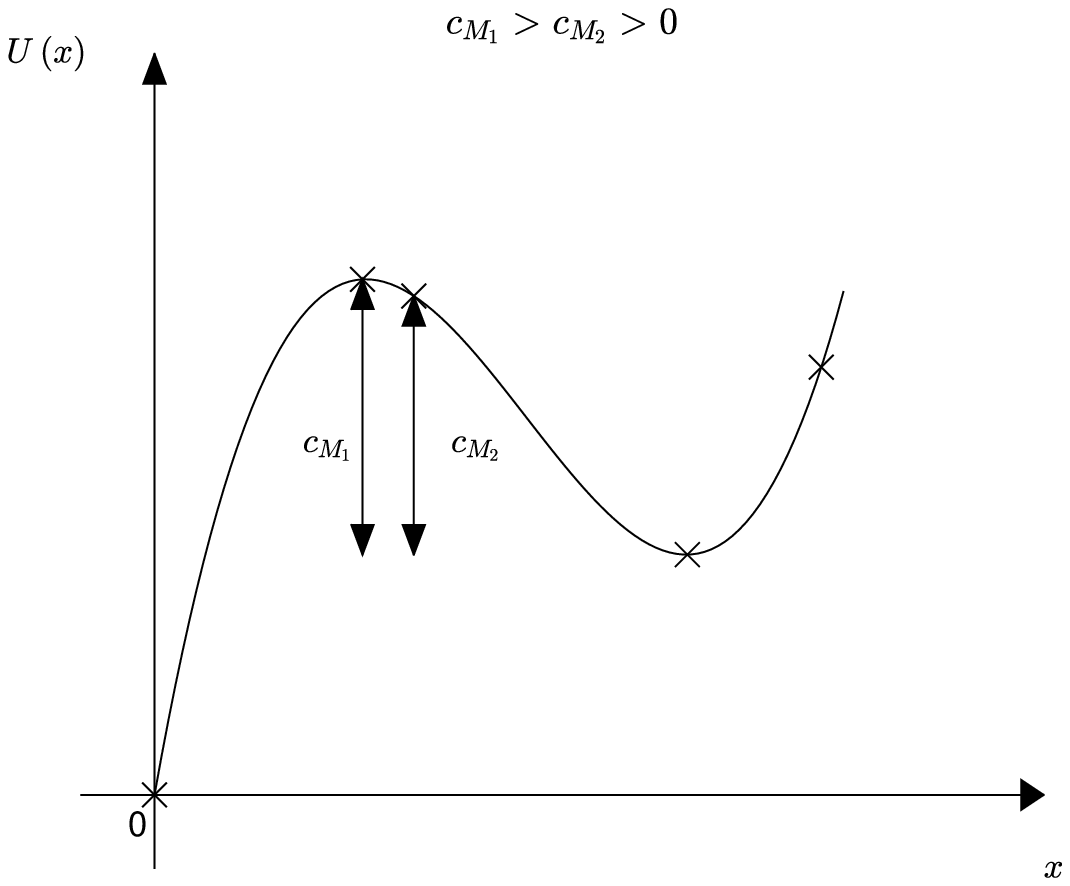} 
		\caption{$c_{M_1} > c_{M_2} > 0$} 
		\label{fig:a} 
		\vspace{4ex}
	\end{subfigure}
	\begin{subfigure}[b]{0.5\linewidth}
		\centering
		\includegraphics[width=1.1\linewidth]{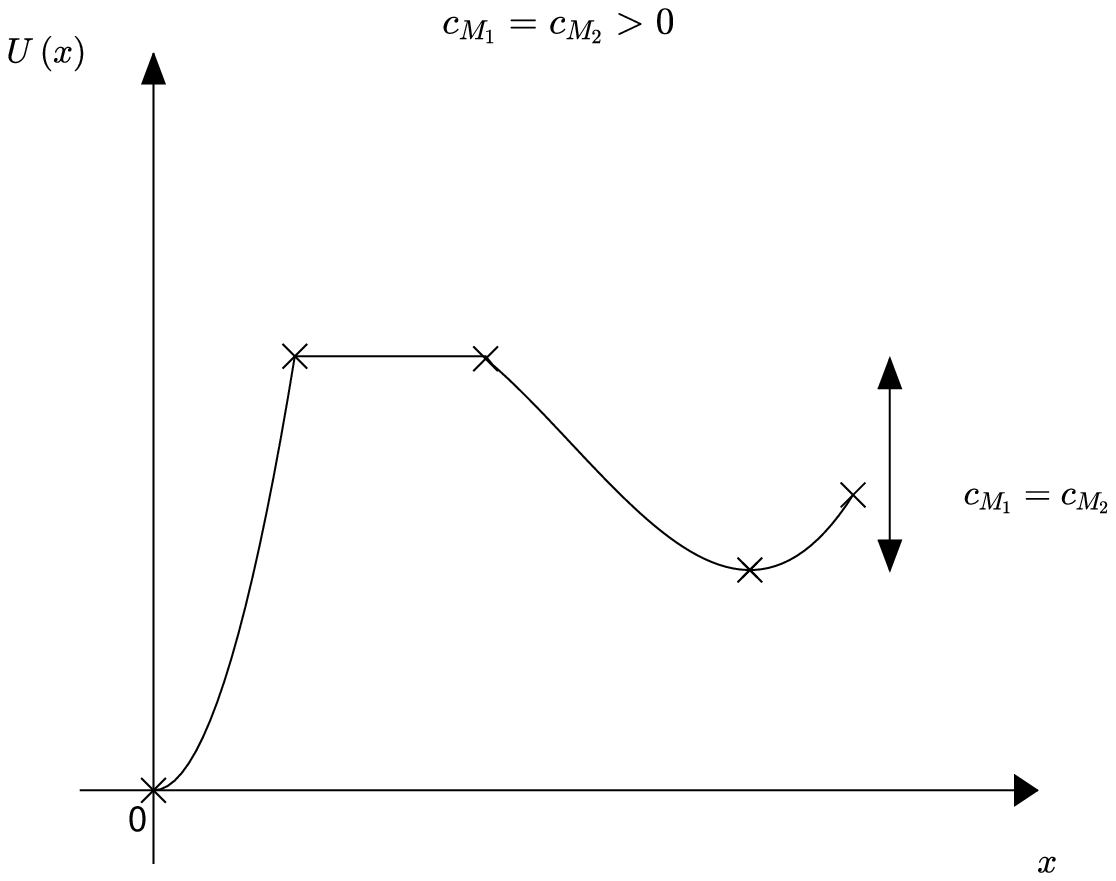} 
		\caption{$c_{M_1} = c_{M_2} > 0$} 
		\label{fig:b} 
		\vspace{4ex}
	\end{subfigure} 
	\begin{subfigure}[b]{0.5\linewidth}
		\centering
		\includegraphics[width=1.1\linewidth]{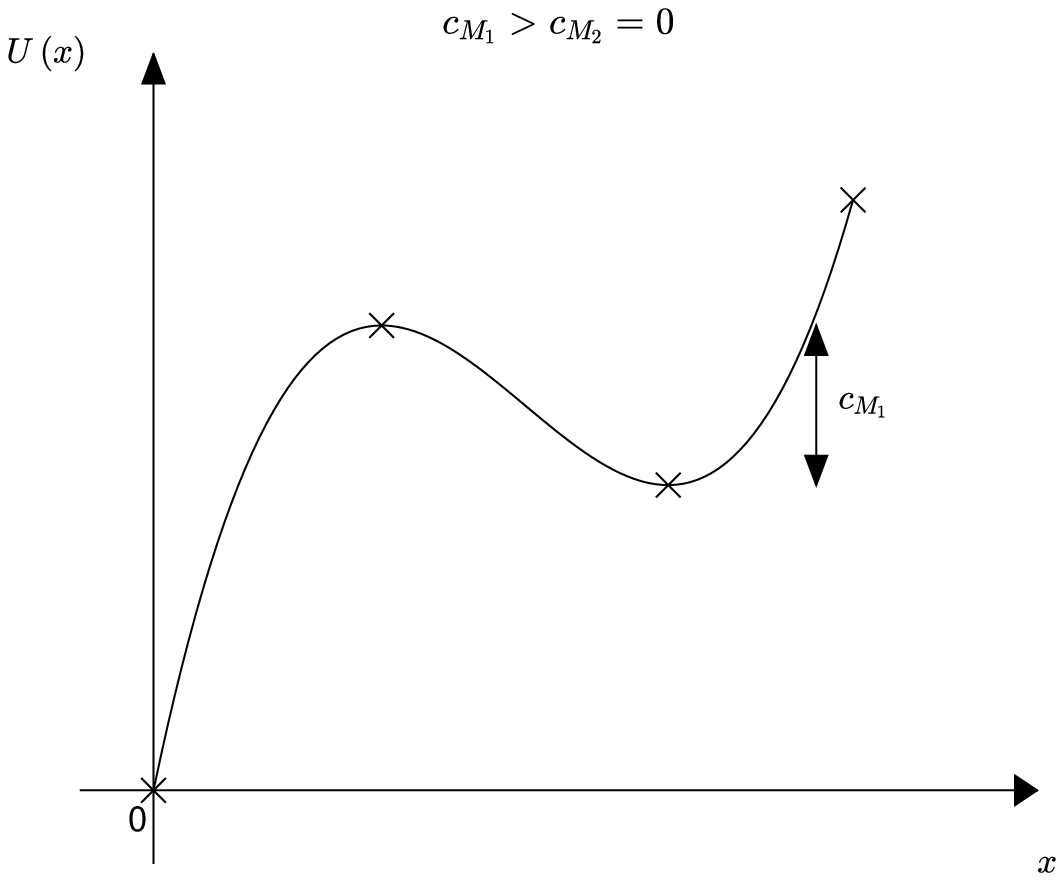} 
		\caption{$c_{M_1} > c_{M_2} = 0$} 
		\label{fig:c} 
	\end{subfigure}
	\begin{subfigure}[b]{0.5\linewidth}
		\centering
		\includegraphics[width=1.1\linewidth]{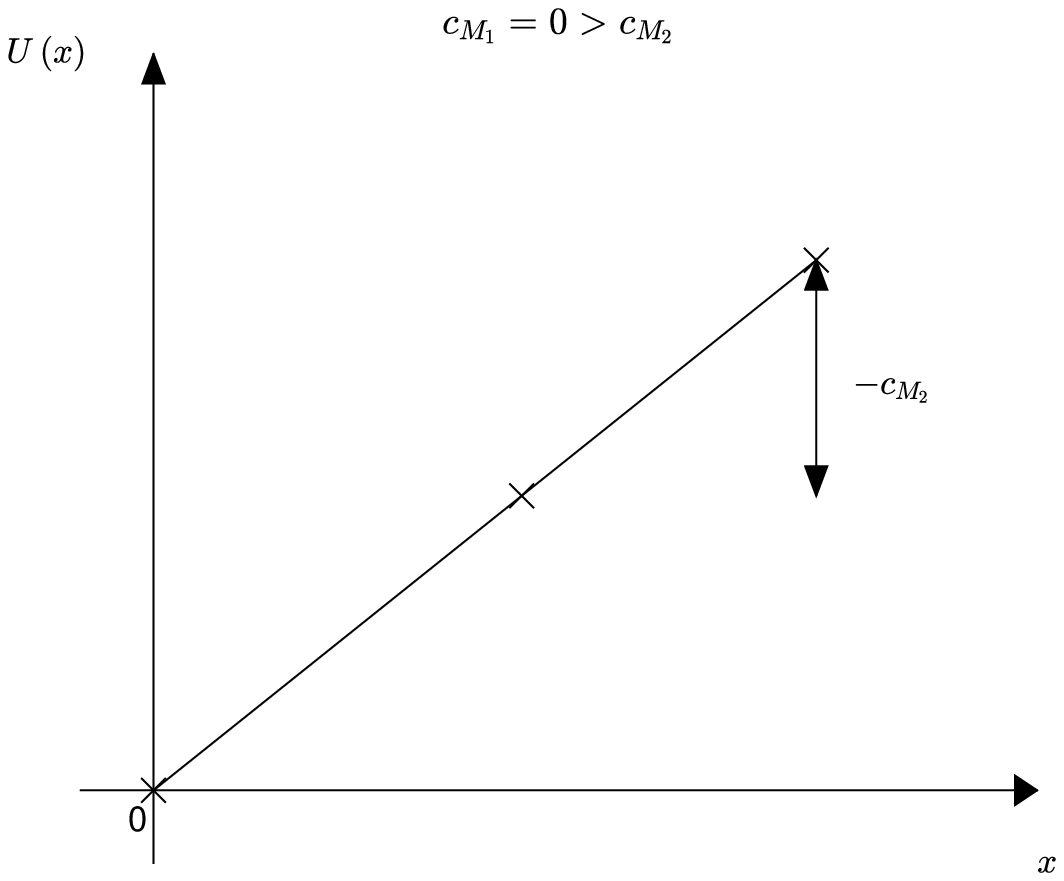} 
		\caption{$c_{M_1} = 0 > c_{M_2}$} 
		\label{fig:d} 
	\end{subfigure} 
	\caption{Illustration of $c_{M_1}$ and $c_{M_2}$ in different situations. We assume that $U$ is one-dimensional consisting of crosses on the graph, and $Q$ is a birth-death process that only explore neighbouring points on the left or on the right.}
	\label{fig:cM1cM2} 
\end{figure}

We recall that for cooling schedule of the form $T(t) = \dfrac{c_{M_1} + \epsilon}{\log (t+1)}$, $\epsilon > 0$, \cite{Gidas85, HS88} show the strong ergodicity of $X^{M_1}$. In the following main result, $c_{M_2}$ plays a similar role as that of $c_{M_1}$ in the cooling schedule. We split our ergodicity results into two cases according to whether $c_{M_2} > 0$ or $c_{M_2} \leq 0$. We first look at the case $c_{M_2} > 0$:

\begin{theorem}[Strong ergodicity of $X^{M_2}$ and convergence in relative entropy under logarithmic cooling and $c_{M_2} > 0$]\label{thm:stronge}
	Let $X^{M_2}$ be the improved variant introduced in Definition \ref{def:M2}, and suppose that $c_{M_2} > 0$. Under the logarithmic cooling schedule of the form 
	$$T(t) = \dfrac{c_{M_2} + \epsilon}{\log(t+1)},$$ 
	where $\epsilon > 0$, then $X^{M_2}$ is 
	\begin{enumerate}
		\item\label{it:stronge} strongly ergodic  and converges to $\pi_{min}$ in total variation distance in the sense of Definition \ref{def:weakestronge};
		\item\label{it:relativee} converges in relative entropy in the sense of Definition \ref{def:relativee}. Note that this implies item \eqref{it:stronge} by the Pinsker's inequality.
	\end{enumerate}
	Thus, the improved variant $X^{M_2}$ converges at a faster logarithmic cooling compared to $X^{M_1}$, and the speed-up depends on the difference $c_{M_1} - c_{M_2}$.
\end{theorem}

At first glance, it is perhaps hard to see the reason why $c_{M_2}$ appears in the cooling schedule for $X^{M_2}$ while $c_{M_1}$ is the corresponding constant for $X^{M_1}$. The reason becomes clear when one looks at the proof of Theorem \ref{thm:stronge}: it relies crucially on a spectral gap lower bound of $M_{2,t}$, where the difference between $c_{M_1}$ and $c_{M_2}$ follows from the difference between the quadratic form of $M_{1,t}$ and $M_{2,t}$ in Lemma \ref{lem:M1M2}. This observation also leads us to the second case when $c_{M_2} \leq 0$. In such case, the spectral gap of $M_{2,t}$ is uniformly bounded away from $0$ for all $t > 0$, and so one expect that faster than logarithmic cooling should work for $X^{M_2}$ in this regime. Our next two results tell us that it is indeed the case:

\begin{theorem}[Strong ergodicity of $X^{M_2}$ under faster than logarithmic cooling and $c_{M_2} \leq 0$]\label{thm:stronge2}
	Let $X^{M_2}$ be the improved variant introduced in Definition \ref{def:M2}, and suppose that $c_{M_2} \leq 0$. If the cooling schedule satisfies
	\begin{align}\label{eq:fastcool}
	\lim_{t \to \infty} \left(\dfrac{d}{dt} T(t)\right) \dfrac{e^{\frac{c_{M_2}}{T(t)}}}{T(t)^2} = 0 ,
	\end{align}
	then $X^{M_2}$ is strongly ergodic  and converges to $\pi_{min}$ in total variation distance in the sense of Definition \ref{def:weakestronge}. If $c_{M_2} = 0$, examples of fast cooling schedule that satisfy \eqref{eq:fastcool} are
	\begin{enumerate}
		\item(power law cooling) $T(t) = (t+1)^{-\alpha}$, where $\alpha \in (0,1)$.
		\item(powers of logarithmic cooling) $T(t) = \left(\log(t+1)\right)^{-k}$, where $k > 1$.
		\item $T(t) = (t+1)^{-\alpha} \left(\log(t+1)\right)^{-1}$, where $\alpha \in (0,1)$.
	\end{enumerate}
	If $c_{M_2} < 0$, examples of fast cooling schedule that satisfy \eqref{eq:fastcool}, in addition to those listed above, are
	\begin{enumerate}[resume]
		\item(exponential cooling) $T(t) = e^{-t}$.
		\item(power law cooling) $T(t) = (t+1)^{-\alpha}$, where $\alpha > 0$.
	\end{enumerate}
	Note that all these examples are faster than logarithmic cooling in the sense that
	$\lim_{t \to \infty} T(t) \log(t+1) = 0$.
\end{theorem}

\begin{theorem}[Convergence in relative entropy of $X^{M_2}$ under faster than logarithmic cooling and $c_{M_2} \leq 0$]\label{thm:relativee2}
	Let $X^{M_2}$ be the improved variant introduced in Definition \ref{def:M2}, and suppose that $c_{M_2} \leq 0$. If the cooling schedule satisfies
	\begin{align}\label{eq:fastcoolrelativee}
	\int_0^{\infty} (1+1/T(t))^{-1} \exp\{-c_{M_2}/T(t)\} \,dt &= \infty, \\
	\lim_{t \to \infty} \left|\dfrac{d}{dt}T(t)\right| \dfrac{\exp\{c_{M_2}/T(t)\}}{T(t)^2 + T(t)^3} &= 0, \label{eq:fastcoolrelativee2}
	\end{align}
	then $X^{M_2}$ converges in relative entropy in the sense of Definition \ref{def:relativee}. If $c_{M_2} = 0$, an example of fast cooling schedule that satisfies \eqref{eq:fastcoolrelativee} and \eqref{eq:fastcoolrelativee2} is
	\begin{enumerate}
		\item(power law cooling) $T(t) = (t+1)^{-\alpha}$, where $\alpha \in (0,1)$.
	\end{enumerate}
	If $c_{M_2} < 0$, an example of fast cooling schedule that satisfies \eqref{eq:fastcoolrelativee} and \eqref{eq:fastcoolrelativee2}, in addition to those listed above, is
	\begin{enumerate}[resume]
		\item(exponential cooling) $T(t) = e^{-t}$.
	\end{enumerate}
	Note that all these examples are faster than logarithmic cooling in the sense that
	$\lim_{t \to \infty} T(t) \log(t+1) = 0$.
\end{theorem}

Note that the speed-up of the proposed variant is most prominent in the case when $c_{M_1} > 0$ while $c_{M_2} \leq 0$, see for instance the cases in Figure \ref{fig:c} and Figure \ref{fig:d}. When $c_{M_1} > 0$, the convergence guarantee of $X^{M_1}$ is given by the logarithmic cooling
$$\dfrac{c_{M_1}}{\log(t+1)},$$
while for $X^{M_2}$ one can apply a fast and non-logarithmic cooling schedule as mentioned in Theorem \ref{thm:stronge2} and Theorem \ref{thm:relativee2}. On the other hand, the convergence guarantee of $X^{M_2}$ reduces to that of $X^{M_1}$ in the case $c_{M_1} = c_{M_2} > 0$ (see for example Figure \ref{fig:b}). 

While the previous two main results are asymptotic in nature and concerns about long-time convergence in total variation and in relative entropy, in our next result we analyze finite-time behaviour of $X^{M_2}$. This is in fact a simple corollary from the general result in \cite{DM94} for non-homogeneous Markov chains:
\begin{corollary}[Finite-time convergence estimate of $X^{M_2}$]\label{thm:finite}
	Let $X^{M_2}$ be the improved variant introduced in Definition \ref{def:M2}. If $c_{M_2} > 0$, then under the logarithmic cooling schedule of the form 
	$$T(t) = \dfrac{c_{M_2}}{\log(\rho t+1)},$$
	where $$\rho := \dfrac{2 c_{M_2} A}{3R}$$
	and we recall $R, K, B$ are defined in \eqref{eq:R} and \eqref{eq:AB} and $A$ is defined in Lemma \ref{lem:spectralgaplower}. The finite-time estimate is then
	\begin{align}\label{eq:finiteM2}
		\mathbb{P}_x(X^{M_2}_t \notin U_{min}) &\leq 5K (\rho t+1)^{-\frac{B}{c_{M_2}}} + \left(\pi(x)^{-1} -1\right)^{1/2} K^{1/2} (\rho t+1)^{-\frac{R + B/2}{c_{M_2}}}.
	\end{align}
\end{corollary}

\begin{rk}
	Recall that in \cite{DM94}, under the cooling schedule 
	$$T_1(t) = \dfrac{c_{M_1}}{\log(\rho_1 t+1)}, \quad \rho_1 := \dfrac{2 c_{M_1} A}{3R},$$
	we have
	\begin{align}\label{eq:finiteM1}
	\mathbb{P}_x(X^{M_1}_t \notin U_{min}) &\leq 5K (\rho_1 t+1)^{-\frac{B}{c_{M_2}}} + \left(\pi(x)^{-1} -1\right)^{1/2} K^{1/2} (\rho_1 t+1)^{-\frac{R + B/2}{c_{M_1}}}.
	\end{align}
	Since $T_1(t)$ is non-decreasing in $c_{M_1}$ and $c_{M_1} \geq c_{M_2}$ by Proposition \ref{prop:comparecm1cm2}, the upper bound in \eqref{eq:finiteM2} is less than or equal to that of \eqref{eq:finiteM1}. Thus we have a tighter estimate on the finite-time performance of $X^{M_2}$.
\end{rk}
There are simple examples in the literature which demonstrate $X^{M_1}$ can get stuck at a local minimum and does not converge to $U_{min}$ under fast cooling, see e.g. \cite[Example $8.10$]{Bremaud99} and \cite[Example $13.4$]{H02}. As a result, one possible concern for $X^{M_2}$ under fast cooling is that it may exhibit similar behaviour as $X^{M_1}$ and get trapped at a local minimum as the temperature cools down quickly. Our next result shows that such concern is not valid. Let us now introduce a few notations:
\begin{align}
	N(x) &:= \{y \in \mathcal{X};~Q(x,y) > 0\}, \label{eq:neighbour}\\
	U_{min}^{loc} &:= \{x \in \mathcal{X};~ U(y) \geq U(x), ~ y \in N(x)\}, \\
	\delta &:= \min \{U(y) - U(x);~ x \in U^{loc}_{min}, y \in N(x)\}.
\end{align}
Here $N(x)$ is the neighbourhood of $x$ induced by the proposal generator $Q$ and $U_{min}^{loc}$ is the set of local minimum of $U$. If all local minima are strict local minima, then $\delta > 0$.

\begin{theorem}[$X^{M_2}$ effectively escapes local minimum while $X^{M_1}$ may get trapped under fast cooling]\label{thm:escape}
	For $i = 1,2$, let $X^{M_i} = (X^{M_i}_t)_{t \geq 0}$ be the non-homogeneous continuous-time Markov chain introduced in Definition \ref{def:M1} and \ref{def:M2} with generator $M_{i,t}$, proposal generator $Q$ and target function $U$. Suppose that $x \in U^{loc}_{min}$ and under any cooling schedule,
		\begin{align}\label{eq:escape1}
		\mathbb{P}_x(X_t^{M_2} = x ~\forall t \geq 0) = \lim_{t \to \infty} e^{-\left(\sum_{y \neq x} Q(x,y)\right)t} = 0.
		\end{align}
	If $\delta > 0$ (e.g. all local minima are strict) and under cooling schedule of the form
	$$T(t) = \dfrac{\delta-\epsilon}{\log(t+1)},$$
	where $\epsilon > 0$ such that $\delta > \epsilon$, then
	\begin{align}\label{eq:escape2}
		\mathbb{P}_x(X_t^{M_1} = x ~\forall t \geq 0) > 0.
	\end{align}
\end{theorem}

\begin{rk}[Basins of attraction and metastable behaviour of $X^{M_2}$]
	In Theorem \ref{thm:escape}, we only investigate and compare the behaviour of $X^{M_1}$ and $X^{M_2}$ escaping a single local minimum. In \cite{Choi20}, we study and compare the rate at which $X^{M_1}$ and $X^{M_2}$ escape the \textit{basins of attraction} of local minima when the temperature goes down to zero. Perhaps unsurprisingly, it turns out that the two hill-climbing constants $c_{M_1}$ and $c_{M_2}$ again play an important role in the analysis.
\end{rk}

\section{Proof of the main results}\label{sec:proof}

In this section we prove the main results presented in Section \ref{sec:main}.

%

\subsection{Proof of Proposition \ref{prop:comparecm1cm2}}

We first prove item \eqref{it:cm1cm2}. Fix $x \neq y \in \mathcal{X}$. Along a path $\gamma^{x,y} = \left(\gamma^{x,y}_i\right)_{i=1}^n$ that achieves the lowest elevation, we have
\begin{align*}
	H_1(x,y) &= \mathrm{Elev}_1(\gamma^{x,y}) \\
		   &= \max\{U(\gamma_i^{x,y});~\gamma_i^{x,y} \in \gamma^{x,y}\} \\
		   &\geq \max\{U(\gamma_{i-1}^{x,y}) \wedge U(\gamma_{i}^{x,y});~\gamma_i^{x,y} \in \gamma^{x,y}, i = 1,2,\ldots,n\} = \mathrm{Elev}_2(\gamma^{x,y}) \geq H_2(x,y).
\end{align*}
The desired result follows from subtracting both sides by $U(x) + U(y)$ and taking maximum over all $x \neq y$. In particular, when $U$ does not have repeated values, then the first inequality above becomes a strict inequality. To see item \eqref{it:cm1nonnegative}, first we note that $c_{M_1}$ is attained at some $x$ and $y_0 \in U_{min}$ according to \cite{HS88}. For any $\gamma \in \Gamma^{x,y_0}$, using the definition of elevation we note that
\begin{align*}
	\mathrm{Elev}_1(\gamma) \geq U(x),
\end{align*}
and so
$$c_{M_1} = H_1(x,y_0) - U(x) - U(y_0) \geq - U(y_0) = 0.$$

Next, we prove item \eqref{it:cm2negative}, which can be seen from the observation that
$$H_2(x,y) \leq \mathrm{Elev}_2(\gamma^{x,y}) \leq s^{x,y}.$$

Finally, we prove item \eqref{it:whencm2negative}. The equivalence between $c_{M_1} = 0$ and $U$ having only $1$ local minimum up to equivalence is proved in \cite[Proposition $3.1$]{I94}. When $U$ only has $1$ local minimum and $U$ does not have repeated values, then using item \eqref{it:cm1cm2} we have $c_{M_2} < c_{M_1} = 0$. To prove the other direction, assume the contrary that we have $c_{M_2} < 0$ and $c_{M_1} > 0$. Without loss of generality, assume that $y$ is a global minimum and we find a path $\gamma^{x,y}$ that attains $c_{M_1} > 0$, and we thus have
$$\max_i U(\gamma^{x,y}_i) > U(x).$$
On the other hand, since $c_{M_2} < 0$, along this selected path we have
$$U(x) \wedge U(\gamma_2^{x,y}) \leq \max _{i \geq 1} U\left(\gamma_{i}^{x, y}\right) \wedge U\left(\gamma_{i+1}^{x, y}\right)<U(x),$$
and we have $U(\gamma_2^{x,y}) < U(x)$. Consider the path $(\gamma^{x,y}_i)_{i=2}^n$:
$$\max_{i \geq 2} U(\gamma^{x,y}_i) - U(\gamma^{x,y}_2) > c_{M_1} = \max_{i \geq 1} U(\gamma^{x,y}_i) - U(x),$$
which contradicts the definition of $c_{M_1}$. Thus we have $c_{M_1} = 0$. Using \cite[Proposition $3.1$]{I94} again, $U$ only has $1$ local minimum up to equivalence.


\subsection{Proof of Theorem \ref{thm:stronge}}

Recall that we assume without loss of generality that $\min_y U(y) = 0$. In Subsection \ref{subsubsec:stronge}, we prove Theorem \ref{thm:stronge} item \eqref{it:stronge}, while in Subsection \ref{subsubsec:relativee}, we prove Theorem \ref{thm:stronge} item \eqref{it:relativee}.

\subsubsection{Proof of Theorem \ref{thm:stronge} item \eqref{it:stronge}}\label{subsubsec:stronge}

Assume for now that we have the following spectral gap lower bound for $M_{2,t}$. We will return to its proof at the end of this Subsection.

\begin{lemma}\label{lem:spectralgaplower}[Spectral gap estimate for $M_{2,t}$]
	There is a constant $A > 0$ such that for $t \geq 0$,
	$$\lambda_2(-M_{2,t}) \geq A \exp\bigg\{- \frac{c_{M_2}}{T(t)}\bigg\},$$
	where we recall that $c_{M_2}$ is defined in \eqref{eq:cm2}.
\end{lemma}

Our plan is to invoke Lemma \ref{lem:stronge} by taking 
\begin{align*}
	Q_t &= M_{2,t}, \quad \mu_t = \pi_{T(t)}, \quad T(t) = \dfrac{c_{M_2} + \epsilon}{\log(t+1)}, \\
	\gamma(t) &= -\left(\dfrac{d}{dt}T(t)\right)\dfrac{1}{T(t)^2} (\max_y U(y) - \min_y U(y)) = \dfrac{1}{(c_{M_2} + \epsilon)(t+1)} (\max_y U(y) - \min_y U(y)),
\end{align*}
where the form of $\gamma(t)$ follows from \cite[equation $(2.18)$]{Gidas85}. Using Lemma \ref{lem:spectralgaplower} twice leads to
\begin{align*}
	\int_0^{\infty} \lambda_2(-M_{2,t}) \, dt &\geq \int_0^{\infty} A \exp\bigg\{- \frac{c_{M_2}}{T(t)}\bigg\} \, dt \\
	&= A \int_0^{\infty} (t+1)^{-\frac{c_{M_2}}{c_{M_2} + \epsilon}} \, dt \\
	&\geq A \int_0^{\infty} \dfrac{1}{t+1} \, dt = \infty, \\
	\lim_{t \to \infty} \dfrac{\gamma(t)}{\lambda_2(-M_{2,t})} &\leq \dfrac{ \max_y U(y) - \min_y U(y)}{A(c_{M_2} + \epsilon)} \lim_{t \to \infty} \dfrac{1}{(t+1)^{\frac{\epsilon}{c_{M_2} + \epsilon}}} = 0.
\end{align*}
It is also well-known that 
$$\lim_{t \to \infty} ||\pi_{T(t)} - \pi_{min}||_{TV} = 0.$$
(This holds for any cooling schedule $T(t)$ that is positive and decreases to $0$ as $t \to \infty$, see e.g. \cite[Chapter $8$ Example $8.6$]{Bremaud99}.) It remains for us to prove Lemma \ref{lem:spectralgaplower}.

\begin{proof}[Proof of Lemma \ref{lem:spectralgaplower}]
	Our proof follows the same strategy as in \cite[Lemma $2.7$]{HS88} and \cite[Theorem $2.1$]{Lowe96}, except that now we are analyzing $M_{2,t}$ with quadratic form given as in Lemma \ref{lem:M1M2}. In view of the variational formula \eqref{eq:spectralgap}, we will prove that for all $f \in \ell^2{(\pi_{T(t)})}$ with $\pi_{T(t)}(f) = 0$,
	$$\dfrac{\langle -M_{2,t}f,f \rangle_{\pi_{T(t)}}}{\langle f,f \rangle_{\pi_{T(t)}}} \geq A \exp\bigg\{- \frac{c_{M_2}}{T(t)}\bigg\}.$$
	For $x,y \in \mathcal{X}$, we pick $\gamma^{x,y}$ that achieves $\mathrm{Elev}_2(\gamma^{x,y}) = H_2(x,y)$. Let $n(x,y)$ be the length of the path $\gamma^{x,y}$, and define
	$$N := \max_{x,y \in \mathcal{X}} n(x,y).$$
	For $z, w \in \mathcal{X}$, we denote the indicator function $\chi_{z,w}$ to be
	\begin{align*}
		\chi_{z,w}(x,y) = \begin{cases}
		1, \quad \text{for some}~ 0 \leq i < n(x,y), ~ \gamma^{x,y}_i = z, \gamma^{x,y}_{i+1} = w, \\
		0, \quad \text{otherwise.}
		\end{cases}
	\end{align*}
	Let $\alpha(x,y) := \pi(x) Q(x,y)$. Reversibility of $Q$ implies $\alpha$ is symmetric. If $\alpha(z,w) = 0$, then $\chi_{z,w}(x,y) = 0$ for all $x,y$. In the following we take $\chi_{z,w}(x,y)/\alpha(z,w) = 0$ if $\chi_{z,w}(x,y) = 0$. We see that
	\begin{align*}
		2 \langle f,f \rangle_{\pi_{T(t)}} &= \sum_{x,y} (f(y)-f(x))^2 \pi_{T(t)}(x) \pi_{T(t)}(y) \\
		&= \sum_{x,y} \left(\sum_{i=1}^{n(x,y)} f(\gamma^{x,y}_{i}) - f(\gamma^{x,y}_{i-1})\right)^2 \pi_{T(t)}(x) \pi_{T(t)}(y) \\
		&\leq \sum_{x,y} n(x,y)\sum_{i=1}^{n(x,y)} (f(\gamma^{x,y}_{i}) - f(\gamma^{x,y}_{i-1}))^2 \pi_{T(t)}(x) \pi_{T(t)}(y) \\
		&\leq N \sum_{x,y} \sum_{w,z} \chi_{z,w}(x,y) (f(z) - f(w))^2 \left(\alpha(z,w) / Z_{T(t)} \right) e^{-\frac{\min\{U(z),U(w)\}}{T(t)}} \dfrac{\pi_{T(t)}(x) \pi_{T(t)}(y) Z_{T(t)}}{\alpha(z,w) e^{-\frac{\min\{U(z),U(w)\}}{T(t)}}} \\
		&\leq N \left(\max_{z,w} \sum_{x,y} \chi_{z,w}(x,y)  \dfrac{\pi_{T(t)}(x) \pi_{T(t)}(y) Z_{T(t)}}{\alpha(z,w) e^{-\frac{\min\{U(z),U(w)\}}{T(t)}}} \right) \\
		&\quad \times \sum_{z,w} (f(z) - f(w))^2 \left(\alpha(z,w) / Z_{T(t)} \right) e^{-\frac{\min\{U(z),U(w)\}}{T(t)}} \\
		&= 2N  \left(\max_{z,w} \sum_{x,y} \chi_{z,w}(x,y)  \dfrac{\pi_{T(t)}(x) \pi_{T(t)}(y) Z_{T(t)}}{\alpha(z,w) e^{-\frac{\min\{U(z),U(w)\}}{T(t)}}} \right) \langle -M_{2,t}f,f \rangle_{\pi_{T(t)}}
	\end{align*}
	where the first inequality follows from Cauchy-Schwartz inequality, the second inequality follows from $n(x,y) \leq N$, and the last equality follows from the quadratic form given in Lemma \ref{lem:M1M2}. Now note that
	\begin{align*}
	 \chi_{z,w}(x,y)  \dfrac{\pi_{T(t)}(x) \pi_{T(t)}(y) Z_{T(t)}}{\alpha(z,w) e^{-\frac{\min\{U(z),U(w)\}}{T(t)}}} &= \dfrac{\chi_{z,w}(x,y)}{\alpha(z,w)} \dfrac{\pi(x)\pi(y)}{Z_{T(t)}} e^{\frac{\min\{U(z),U(w)\} - U(x) - U(y)}{T(t)}} \\
	 &\leq \exp\bigg\{\frac{c_{M_2}}{T(t)}\bigg\} \dfrac{\chi_{z,w}(x,y)}{\alpha(z,w)} \dfrac{\pi(x)\pi(y)}{\pi(U_{min})}.
	\end{align*}
	Desired result follows by taking 
	$$A^{-1} = N \left(\max_{z,w} \sum_{x,y} \dfrac{\chi_{z,w}(x,y)}{\alpha(z,w)} \dfrac{\pi(x)\pi(y)}{\pi(U_{min})}\right).$$
\end{proof}

\subsubsection{Proof of Theorem \ref{thm:stronge} item \eqref{it:relativee}}\label{subsubsec:relativee}

Our plan is to invoke Lemma \ref{lem:relativee} and identify the appropriate $a_{T(t)}$. The spectral gap estimate of $M_{2,t}$ in Lemma \ref{lem:spectralgaplower} will again play a crucial role. First, we compute 

\begin{align}\label{eq:dentropy}
	\dfrac{d}{dt} \mathrm{Ent}_{\pi_{T(t)}}\left(P_{0,t}^{M_2}(x,\cdot)\right) = \sum_{y} \left(\log \left(\dfrac{P_{0,t}^{M_2}(x,y)}{\pi_{T(t)}(y)} \right) \dfrac{d}{dt} P_{0,t}^{M_2}(x,y) - \dfrac{P_{0,t}^{M_2}(x,y)}{\pi_{T(t)}(y)} \dfrac{d}{dt} \pi_{T(t)}(y)\right).
\end{align}

Using the Kolmogorov forward equation we see that

\begin{align*}
\dfrac{d}{dt} P_{0,t}^{M_2}(x,y) &= \sum_k P_{0,t}^{M_2}(x,k) M_{2,t}(k,y) \\
								 &= \sum_{k \neq y} M_{2,t}(k,y) P_{0,t}^{M_2}(x,k) - \sum_{k \neq y} P_{0,t}^{M_2}(x,y) M_{2,t}(y,k) \\
								 &= \sum_{k \neq y} \pi_{T(t)}(k) M_{2,t}(k,y) \left(\dfrac{P_{0,t}^{M_2}(x,k)}{\pi_{T(t)}(k)} - \dfrac{P_{0,t}^{M_2}(x,y)}{\pi_{T(t)}(y)} \right) \\
								 &= \sum_{k \neq y} \pi_{T(t)}(k) M_{2,t}(k,y) \left(f(k) - f(y) \right), \\
\end{align*}
where in the last equality we let 
\begin{align}\label{eq:f}
	f(k) := \dfrac{P_{0,t}^{M_2}(x,k)}{\pi_{T(t)}(k)}.
\end{align}

Substituting these back into \eqref{eq:dentropy} yields
\begin{align} \label{eq:immediate}
\dfrac{d}{dt} \mathrm{Ent}_{\pi_{T(t)}}\left(P_{0,t}^{M_2}(x,\cdot)\right) &= \sum_{y} \log f(y) \sum_{k \neq y} \pi_{T(t)}(k) M_{2,t}(k,y) \left(f(k) - f(y) \right) - \sum_y  \dfrac{P_{0,t}^{M_2}(x,y)}{\pi_{T(t)}(y)} \dfrac{d}{dt} \pi_{T(t)}(y). \nonumber \\
&\leq -2 C \sum_{k,y}  \pi_{T(t)}(k) M_{2,t}(k,y) \left(\sqrt{f}(k) - \sqrt{f}(y)\right)^2 - \sum_y  \dfrac{P_{0,t}^{M_2}(x,y)}{\pi_{T(t)}(y)} \dfrac{d}{dt} \pi_{T(t)}(y) \nonumber\\
&= -4 C \langle -M_{2,t}\sqrt{f},\sqrt{f} \rangle_{\pi_{T(t)}}  - \sum_y \dfrac{P_{0,t}^{M_2}(x,y)}{\pi_{T(t)}(y)} \dfrac{d}{dt} \pi_{T(t)}(y),
\end{align}
where the first inequality follows from \cite[Proposition $3$]{Miclo92} and $C$ is a constant that depends on the size of $\mathcal{X}$. Now, for this choice of $f$ in \eqref{eq:f}, $p \in (2, \infty)$ and Lemma \ref{lem:spectralgaplower},
\begin{align*}
	||\sqrt{f} - \pi(\sqrt{f}) ||_{\ell^p(\pi)}^2 &\leq ||\sqrt{f} - \pi(\sqrt{f}) ||_{\ell^2(\pi)}^2 \leq \dfrac{\langle -Q \sqrt{f}, \sqrt{f} \rangle_{\pi}}{\lambda_2(-Q)}, \\
	||\sqrt{f} - \pi_{T(t)}(\sqrt{f}) ||_{\ell^2(\pi_{T(t)})}^2 &\leq \dfrac{\langle -M_{2,t} \sqrt{f}, \sqrt{f} \rangle_{\pi_{T(t)}}}{\lambda_2(-M_{2,t})} \leq A \exp\bigg\{ \frac{c_{M_2}}{T(t)}\bigg\} \langle -M_{2,t} \sqrt{f}, \sqrt{f} \rangle_{\pi_{T(t)}}, 
\end{align*}
where $A$ is the constant that appears in Lemma \ref{lem:spectralgaplower}. As a result, equation $(3.12)$ and $(3.15)$ in \cite{HS88} are fulfilled. By \cite[Theorem $3.21$]{HS88}, there is a constant $D = D(p,c_{M_2},\lambda_2(-Q),\max_y U(y) - \min_y U(y)) < \infty$ such that
\begin{align}\label{eq:entropyspectral}
	\langle -M_{2,t}\sqrt{f},\sqrt{f} \rangle_{\pi_{T(t)}} \geq \dfrac{(1+1/T(t))^{-1} \exp\{-c_{M_2}/T(t)\}}{D} \mathrm{Ent}_{\pi_{T(t)}}\left(P_{0,t}^{M_2}(x,\cdot)\right).
\end{align}
Combining \eqref{eq:entropyspectral} and \eqref{eq:immediate} yields
\begin{align*}
	\dfrac{d}{dt} \mathrm{Ent}_{\pi_{T(t)}}\left(P_{0,t}^{M_2}(x,\cdot)\right) &\leq - \dfrac{4C(1+1/T(t))^{-1} \exp\{-c_{M_2}/T(t)\}}{D} \mathrm{Ent}_{\pi_{T(t)}}\left(P_{0,t}^{M_2}(x,\cdot)\right) \\
	&\quad + (\max_yU(y) - \min_y U(y))\left|\dfrac{d}{dt}T(t)\right| \dfrac{1}{T(t)^2},
\end{align*}
where the second term follows from \cite[equation $2.18$]{Gidas85}. This is exactly the form presented in Lemma \ref{lem:relativee} with the choice of 
\begin{align*}
	a_{T(t)} &= \dfrac{D (1+1/T(t)) \exp\{c_{M_2}/T(t)\}}{4C}, \quad R = \max_yU(y) - \min_y U(y).
\end{align*}
Under the logarithmic cooling schedule of the form $T(t) = \dfrac{c_{M_2} + \epsilon}{\log(t+1)}$, we calculate that
\begin{align*}
	\int_0^{\infty} a_{T(t)}^{-1}\,dt &\geq \dfrac{4C}{D} \int_0^{\infty} \dfrac{c_{M_2} + \epsilon}{c_{M_2} + \epsilon + \log(t+1)} \dfrac{1}{t+1} \, dt = \infty. \\
	\lim_{t \to \infty} \left|\dfrac{d}{dt}T(t)\right| \dfrac{a_{T(t)}}{T(t)^2} &= \lim_{t \to \infty} \dfrac{D(c_{M_2} + \epsilon + \log(t+1))}{4C (c_{M_2} + \epsilon)^2 (t+1)^{1+\frac{c_{M_2}}{c_{M_2}+\epsilon}}} = 0
\end{align*}
Desired result follows from Lemma \ref{lem:relativee}.

\subsection{Proof of Theorem \ref{thm:stronge2}}

Similar to Subsection \ref{subsubsec:stronge}, our plan is to invoke Lemma \ref{lem:stronge} by taking 
\begin{align*}
Q_t &= M_{2,t}, \quad \mu_t = \pi_{T(t)}, \\
\gamma(t) &= -\left(\dfrac{d}{dt}T(t)\right)\dfrac{1}{T(t)^2} (\max_y U(y) - \min_y U(y)) = \dfrac{1}{(c_{M_2} + \epsilon)(t+1)} (\max_y U(y) - \min_y U(y)),
\end{align*}
Using Lemma \ref{lem:spectralgaplower} leads to
\begin{align*}
\int_0^{\infty} \lambda_2(-M_{2,t}) \, dt &\geq \int_0^{\infty} A \exp\bigg\{- \frac{c_{M_2}}{T(t)}\bigg\} \, dt \\
&\geq A \int_0^{\infty} 1 \, dt \\
&= \infty,
\end{align*}
where we use the assumption that $c_{M_2} \leq 0$ in the second inequality above. It remains to check
$$\lim_{t \to \infty} \dfrac{\gamma(t)}{\lambda_2(-M_{2,t})} \leq -A^{-1}(\max_y U(y) - \min_y U(y))\lim_{t \to \infty} \left(\dfrac{d}{dt} T(t)\right) \dfrac{e^{\frac{c_{M_2}}{T(t)}}}{T(t)^2}= 0,$$
where we use again Lemma \ref{lem:spectralgaplower} in the first inequality, and the equality follows from the assumption on the cooling schedule \eqref{eq:fastcool}. We proceed to show the cooling schedules proposed satisfy \eqref{eq:fastcool}:
\begin{enumerate}
	\item $T(t) = (t+1)^{-\alpha}$ 
	\begin{align*}
		\lim_{t \to \infty} \left(\dfrac{d}{dt} T(t)\right) \dfrac{1}{T(t)^2} = \lim_{t \to \infty} \dfrac{-\alpha}{(1+t)^{1-\alpha}} = 0.
	\end{align*}
	\item $T(t) = \left(\log(t+1)\right)^{-k}$
	\begin{align*}
	\lim_{t \to \infty} \left(\dfrac{d}{dt} T(t)\right) \dfrac{1}{T(t)^2} = \lim_{t \to \infty} \dfrac{-k (\log(t+1))^{k-1}}{t+1} = 0.
	\end{align*}
	\item $T(t) = (t+1)^{-\alpha} \left(\log(t+1)\right)^{-1}$
	\begin{align*}
	\lim_{t \to \infty} \left(\dfrac{d}{dt} T(t)\right) \dfrac{1}{T(t)^2} = \lim_{t \to \infty} - \dfrac{1}{(1+t)^{1-\alpha}} - \dfrac{\alpha \log(t+1)}{(1+t)^{1-\alpha}} = 0.
	\end{align*}
	\item $T(t) = e^{-t}$
	\begin{align*}
	\lim_{t \to \infty} \left(\dfrac{d}{dt} T(t)\right) \dfrac{e^{\frac{c_{M_2}}{T(t)}}}{T(t)^2} = \lim_{t \to \infty} - e^{t} e^{c_{M_2} e^{t}} = 0,
	\end{align*}
	where we use $c_{M_2} < 0$ in the last equality.
	\item $T(t) = (t+1)^{-\alpha}$
	\begin{align*}
	\lim_{t \to \infty} \left(\dfrac{d}{dt} T(t)\right) \dfrac{e^{\frac{c_{M_2}}{T(t)}}}{T(t)^2} = \lim_{t \to \infty} - \alpha(t+1)^{\alpha-1} e^{c_{M_2}(t+1)^{\alpha}} = 0,
	\end{align*}
	where we use $c_{M_2} < 0$ in the last equality.
\end{enumerate}

\subsection{Proof of Theorem \ref{thm:relativee2}}

Our proof follows from the proof of Theorem \ref{thm:stronge} item \eqref{it:relativee} in Subsection \ref{subsubsec:relativee}. Recall that we have showed
\begin{align*}
\dfrac{d}{dt} \mathrm{Ent}_{\pi_{T(t)}}\left(P_{0,t}^{M_2}(x,\cdot)\right) &\leq - a_{T(t)}^{-1} \mathrm{Ent}_{\pi_{T(t)}}\left(P_{0,t}^{M_2}(x,\cdot)\right) + R\left|\dfrac{d}{dt}T(t)\right| \dfrac{1}{T(t)^2}, \\
a_{T(t)} &= \dfrac{D (1+1/T(t)) \exp\{c_{M_2}/T(t)\}}{4C}, \quad R = \max_yU(y) - \min_y U(y).\\
\end{align*}
Equations \eqref{eq:fastcoolrelativee} and \eqref{eq:fastcoolrelativee2} lead to
\begin{align*}
\int_0^{\infty} a_{T(t)}^{-1} \,dt &= \infty, \quad
\lim_{t \to \infty} \left|\dfrac{d}{dt}T(t)\right| \dfrac{a_{T(t)}}{T(t)^2} = 0.
\end{align*}
Desired result follows from Lemma \ref{lem:relativee}. We proceed to show the cooling schedules proposed satisfy \eqref{eq:fastcoolrelativee} and \eqref{eq:fastcoolrelativee2}:
\begin{enumerate}
	\item $T(t) = (t+1)^{-\alpha}$, $\alpha \in (0,1)$,
	\begin{align*}
	\int_0^{\infty} (1+1/T(t))^{-1} \exp\{-c_{M_2}/T(t)\} \,dt &= \int_0^{\infty} \dfrac{1}{1+(1+t)^{\alpha}} \, dt \geq \int_0^{\infty} \dfrac{1}{2+t} \, dt = \infty, \\
	\lim_{t \to \infty} \left|\dfrac{d}{dt}T(t)\right| \dfrac{\exp\{c_{M_2}/T(t)\}}{T(t)^2 + T(t)^3} &= \lim_{t \to \infty} \dfrac{\alpha}{(1+t)^{1-\alpha} + (1+t)^{1-2\alpha}} = 0, 
	\end{align*}
	\item $T(t) = e^{-t}$
	\newline	
	First, we note that the integrand 
	$$(1+1/T(t))^{-1} \exp\{-c_{M_2}/T(t)\} = \dfrac{e^{-c_{M_2}e^t}}{1+e^t} \to \infty \quad \text{as}~ t \to \infty,$$
	where we use $c_{M_2} < 0$. As a result,
	\begin{align*}
	\int_0^{\infty} (1+1/T(t))^{-1} \exp\{-c_{M_2}/T(t)\} \,dt &= \infty, \\
	\lim_{t \to \infty} \left|\dfrac{d}{dt}T(t)\right| \dfrac{\exp\{c_{M_2}/T(t)\}}{T(t)^2 + T(t)^3} &= \lim_{t \to \infty} \dfrac{e^{c_{M_2}e^t}}{e^{-t} + e^{-2t}} = 0. 
	\end{align*}
%
\end{enumerate}

\subsection{Proof of Corollary \ref{thm:finite}}

The result follows immediately from applying \cite[Corollary $2.2.8$]{DM94}, once we check the conditions there are satisfied. \cite[Equation $(2.2.2)$]{DM94} is satisfied with our $R$, while \cite[Equation $(2.2.6)$]{DM94} is satisfied with our $K$ and $B$, see also \cite[Example $3.1.9$]{DM94}. Finally, \cite[Equation $(2.2.1)$]{DM94} is fulfilled in view of our spectral gap lower bound in Lemma \ref{lem:spectralgaplower}.

\subsection{Proof of Theorem \ref{thm:escape}}

We first prove \eqref{eq:escape1}. When $x \in U^{loc}_{min}$, then $(U(x) - U(y))_+ = 0$ for $y \in N(x)$, and so by the definition of $M_{2,s}$ in Definition \ref{def:M2} we have, for $s \geq 0$,
\begin{align*}
	M_{2,s}(x,y) &= Q(x,y), \\
	M_{2,s}(x,x) &= - \sum_{y \neq x} Q(x,y),
\end{align*}
where we note that the right hand side of these two equations are independent of $s$. According to \cite[equation $(2.6)$]{vDN92}
\begin{align*}
	\mathbb{P}_x(X_s^{M_2} = x ~\forall s \in [0,t]) = \exp\bigg\{\int_0^{t} M_{2,s}(x,x)\, ds \bigg\} = \exp\bigg\{ - \sum_{y \neq x} Q(x,y) t \bigg\}.
\end{align*}
Taking $t \to \infty$ gives \eqref{eq:escape1}. Next, we prove \eqref{eq:escape2}. Using the definition of $\delta > 0$ and $M_{1,t}$, we see that for $x \in U^{loc}_{min}, y \in N(x)$,
\begin{align*}
M_{1,s}(x,y) &\leq e^{-\frac{\delta}{T(s)}}Q(x,y), \\
M_{1,s}(x,x) &\geq - \sum_{y \neq x} e^{-\frac{\delta}{T(s)}}Q(x,y).
\end{align*}
It follows again from \cite[equation $(2.6)$]{vDN92} that
\begin{align*}
\mathbb{P}_x(X_s^{M_1} = x ~\forall s \in [0,t]) = \exp\bigg\{\int_0^{t} M_{1,s}(x,x)\, ds \bigg\} &\geq \exp\bigg\{ - \left(\sum_{y \neq x} Q(x,y) \right) \int_0^{t} e^{-\frac{\delta}{T(s)}}\, ds\bigg\} \\
&= \exp\bigg\{ - \left(\sum_{y \neq x} Q(x,y) \right) \int_0^{t} \dfrac{1}{(s+1)^{\frac{\delta}{\delta-\epsilon}}}\, ds\bigg\}.
\end{align*}
Taking $t \to \infty$, the equation becomes
\begin{align*}
\mathbb{P}_x(X_s^{M_1} = x ~\forall s \geq 0) \geq \exp\bigg\{ - \left(\sum_{y \neq x} Q(x,y) \right) \int_0^{\infty} \dfrac{1}{(s+1)^{\frac{\delta}{\delta-\epsilon}}}\, ds\bigg\} > 0,
\end{align*}
where we use the fact that $\int_0^{\infty} \dfrac{1}{(s+1)^{\frac{\delta}{\delta-\epsilon}}}\, ds < \infty$ as $$\sum_{k=1}^{\infty} \dfrac{1}{k^{\frac{\delta}{\delta-\epsilon}}} < \infty.$$

\section{Two algorithms to simulate $X^{M_2}$}\label{sec:algo}

The objective of this section is to present two algorithms for simulating $X^{M_2}$. On one hand, $X^{M_1}$ is easy to simulate by means of acceptance-rejection: the proposal chain $Q$ proposes a move which is accepted according to certain probability that depends on the target function $U$ and temperature $T(t)$. Such acceptance-rejection procedure seems to be hard to adapt to the setting of $X^{M_2}$, as we are taking $\max$ instead of $\min$ in its definition. As simulating $X^{M_2}$ boils down to simulating a non-homogeneous Markov chain, the algorithms proposed below are two among many possible ways that will serve the purpose.

In the first proposed algorithm, we simulate $X^{M_2}$ directly. For $x \in \mathcal{X}$ and $y \in N(x)$, the neighbours of $x$ as introduced in \eqref{eq:neighbour}, denote $(E_y)_{y \in N(x)}$ to be independent and identically distributed exponential random variables with mean $1$. Starting from time $t_0 \geq 0$, for each $y$ we solve for $T_y$ in 
\begin{align}\label{eq:algo1}
\int_{t_0}^{T_y} M_{2,s}(x,y)\,ds = E_y.
\end{align}
Let $T_{y^*} = \min_{y \in N(x)}{T_y}$, then $X^{M_2}$ will jump from state $x$ at time $t_0$ to state $y^*$ at time $T_{y^*}$. For example, if we adapt the logarithmic cooling schedule $T(t) = c/\ln(t+1)$ with $c > c_{M_2}$, then upon solving the integration in \eqref{eq:algo1} $T_y$ satisfies
$$\dfrac{Q(x,y) \left((T_y+1)^{\frac{(U(x)-U(y))_+}{c}+1} - (t_0+1)^{\frac{(U(x)-U(y))_+}{c}+1}\right)}{\frac{(U(x)-U(y))_+}{c}+1} = E_y,$$
which gives
$$T_y = \left(E_y \dfrac{\left(\frac{(U(x)-U(y))_+}{c}+1\right)}{Q(x,y)} + (t_0+1)^{\frac{(U(x)-U(y))_+}{c}+1}\right)^{\frac{1}{\frac{(U(x)-U(y))_+}{c}+1}} - 1.$$

We now formally state the first proposed algorithm:

\begin{algorithm}[H]\label{algo:M21}
	\DontPrintSemicolon
	
	\KwInput{Proposal reversible generator $Q$, cooling schedule $T(t)$, target function $U$, initial state $x$}
	Sample $(E_y)_{y \in N(x)}$ to be independent and identically distributed exponential random variables with mean $1$.\;
	For each $y \in N(x)$, compute $T_y$ as in \eqref{eq:algo1}.\;
	Let $T_{y^*} = \min_{y \in N(x)}{T_y}$, then $X^{M_2}$ will jump from state $x$ at time $t_0$ to state $y^*$ at time $T_{y^*}$.\;
	Replace $t_0$ by $T_{y^*}$ and $x$ by $y^*$. Repeat.\;
	\caption{An algorithm to simulate $(X^{M_2}_t)_{t = t_0}^{t_1}$}
\end{algorithm}

For Algorithm \ref{algo:M21}, for it to be computationally feasible it requires the number of neighbours $|N(x)|$ of a state $x$ to be small. In addition, it requires solving equation \eqref{eq:algo1} which may not be tractable for some cooling schedules. Also, when compared with the acceptance-rejection procedure of $X^{M_1}$, the computational cost is increased since we need to evaluate $U(y)$ for all $y \in N(x)$.

In the second proposed algorithm, we resort to the idea of uniformization of non-homogeneous Markov chain as introduced in \cite{vDN92}. Roughly speaking, a non-homogeneous continuous-time Markov chain can be thought as a discrete-time non-homogeneous Markov chain time-changed by an associated Poisson process. This allows us to simulate $X^{M_2}$ by its discrete-time counterpart and its affiliated Poisson process. 

Suppose that we would like to simulate $X^{M_2} = (X^{M_2}_t)_{t = t_0}^{t_1}$  with transition semigroup $P^{M_2} = (P^{M_2}_{s,t})_{0 \leq s \leq t}$ on the interval $(t_0,t_1)$. As we assume the cooling schedule $T(t)$ is decreasing toward $0$, by recalling that $R = \max_x U(x) - \min_x U(x)$ as in \eqref{eq:R},
the transition rate at any given state $x$ and any $t \in (t_0,t_1)$ is bounded above by
\begin{align}\label{eq:M}
|M_{2,t}(x,x)| \leq e^{\frac{R}{T(t_1)}} \sum_{y \neq x} Q(x,y) \leq e^{\frac{R}{T(t_1)}} \max_x |Q(x,x)| =: M.
\end{align}
Now, note that 
\begin{align}\label{eq:uniformizedM2}
\mathcal{P}^{M_2}_t := I + \dfrac{1}{M} M_{2,t}
\end{align}
is a valid reversible (with respect to $\pi_{T(t)}$) stochastic matrix. According to \cite[Theorem $3.1$]{vDN92}, we have

\begin{theorem}[Continuous-time non-homogeneous Markov chain as time-changed discrete-time non-homogeneous Markov chain]\label{thm:uniform}
	\begin{align*}
		P^{M_2}_{t_0,t_1} = \sum_{k=0}^{\infty} \dfrac{e^{-(t_1-t_0)M}\left((t_1-t_0)M \right)^k}{k!} \int_{t_0}^{t_1} \ldots \int_{t_0}^{t_1} \mathcal{P}^{M_2}_{n_1} \mathcal{P}^{M_2}_{n_2} \ldots \mathcal{P}^{M_2}_{n_k} \, d \overline{H}(n_1,\ldots,n_k),
	\end{align*}
	where $d \overline{H}(n_1,\ldots,n_k)$ is the density of the order statistics $n_1 \leq n_2 \leq \ldots \leq n_k$ of a $k$-dimensional uniform distribution in $[t_0,t_1] \times \ldots \times [t_0,t_1]$, and $\mathcal{P}^{M_2}_{n_1} \mathcal{P}^{M_2}_{n_2} \ldots \mathcal{P}^{M_2}_{n_k}$ is the standard matrix product.
\end{theorem}

We now formally state the second proposed algorithm:

\begin{algorithm}[H]\label{algo:M2}
	\DontPrintSemicolon
	
	\KwInput{Proposal reversible generator $Q$, cooling schedule $T(t)$, target function $U$, an estimate $R$ of the range of $U$ as in \eqref{eq:R}, an interval $(t_0,t_1)$ to simulate and an initial state $x_0$}
	Calculate $M_{2,t}$ as in Definition \ref{def:M2},, $M$ as in \eqref{eq:M} and $\mathcal{P}^{M_2}_t$ as in \eqref{eq:uniformizedM2}.\;
	Sample a number $N$ from the Poisson distribution with mean $(t_1-t_0)M$.\;
	Draw $N$ random numbers uniformly distributed on the interval $(t_0,t_1)$, and sort them in ascending order. Label them as $n_1 \leq n_2 \leq \ldots \leq n_N$.\;
	\For{$i = 1;\ i \leq N;\ i = i + 1$}{
		With initial state $x_{i-1}$, simulate one-step of a Markov chain with matrix $\mathcal{P}^{M_2}_{n_i}$.\;
		Set $x_i$ to be the state after one-step.\;
	}
	\KwOutput{$x_N$}
	\caption{An algorithm to simulate $(X^{M_2}_t)_{t = t_0}^{t_1}$ by uniformization}
\end{algorithm}

One obvious drawback in applying Algorithm \ref{algo:M2} is that it requires knowledge on estimating the range $R$ of the target function $U$. Such knowledge is not required in simulating $X^{M_1}$, which is simply an acceptance-rejection procedure. Another disadvantage is that it requires simulating a step of the Markov chain with matrix $\mathcal{P}^{M_2}_{n_i}$. While it is perhaps feasible with nearest-neighbour or single-spin type proposal chain $Q$, this step may be computationally expensive when the state space is exponentially large.

\section{Epilogue: Is $X^{M_2}$ really an acceleration?}\label{sec:epilogue}

As emphasized in Section \ref{sec:main}, the speed-up of $X^{M_2}$, in the sense of adapting a faster cooling schedule, is prominent in the case of $c_{M_2} = 0 < c_{M_1}$, see Theorem \ref{thm:stronge2}. As an illustration of this case, we recall the plot in Figure \ref{fig:c}. In this section, we will mention two shortcomings in using $X^{M_2}$.

Perhaps the most important drawback of $X^{M_2}$ lies in its operations complexitiy. In our analysis so far, we have yet to consider the number of operations or the number of jumps of the Markov chains needed to reach a certain level of approximation of the set of global minima. Recall that in Definition \ref{def:M2} we have
$$M_{2,t}(x,y) = e^{\frac{|U(x)-U(y)|}{T(t)}} M_{1,t}(x,y).$$
Up to a fixed time horizon $T$, in view of the above equation we expect that on average $X^{M_2}$ will jump much more often (meaning more function evaluations and computations of $U$) than its counterpart $X^{M_1}$. This calls for a more rigorous analysis on the operations complexity of $X^{M_2}$, and we will leave that as a future research direction. We now present a simple result comparing the probability of reaching a global minimum in a finite time $t$ while keeping both $X^{M_1}$ and $X^{M_2}$ to have the same number of jumps. For $i = 1,2$, we write $N^{M_i}(t)$ to be the number of jumps up until time $t$ of $X^{M_i}$.
\begin{proposition}
	For $i = 1,2$, let $X^{M_i} = (X^{M_i}_t)_{t \geq 0}$ be the non-homogeneous continuous-time Markov chain introduced in Definition \ref{def:M1} and \ref{def:M2} with generator $M_{i,t}$, proposal generator $Q$, target function $U$ and cooling schedule $T_i(t)$. For any $x,y \in \mathcal{X}$, let $k = k(x,y)$ be the number of steps in the shortest path from $x$ to $y$ with distinct elements, that is, $k = \min_{\gamma \in \Gamma^{x,y}} |\gamma|$, then for any $t \geq 0$,
	\begin{align}
		\mP_x\left(X^{M_2}(t) = y | N^{M_2}(t) = k\right) \geq \mP_x\left(X^{M_1}(t) = y | N^{M_1}(t) = k\right).
	\end{align}
	In particular, this result holds for any $x \in \mathcal{X}$ and $y \in U_{min}$.
\end{proposition}

\begin{proof}
	We recall the notations introduced in Theorem \ref{thm:uniform}. Let
	\begin{align*}
		M &:= e^{\frac{R}{T_2(t)}} \max_x |Q(x,x)|, \\
		\mathcal{P}^{M_2}_t &:= I + \dfrac{1}{M} M_{2,t}, \\
		\mathcal{P}^{M_1}_t &:= I + \dfrac{1}{M} M_{1,t}.
	\end{align*}
	Let $\gamma^{x,y} = (\gamma^{x,y}_i)_{i=0}^k$ be the shortest path(s) from $x$ to $y$ with distinct elements. In view of Theorem \ref{thm:uniform} and \cite[Theorem $3.1$]{vDN92}, we have
	\begin{align*}
	\mP_x\left(X^{M_2}(t) = y | N^{M_2}(t) = k\right) &= \int_{0}^{t} \ldots \int_{0}^{t} \mathcal{P}^{M_2}_{n_1} \mathcal{P}^{M_2}_{n_2} \ldots \mathcal{P}^{M_2}_{n_k}(x,y) \, d \overline{H}(n_1,\ldots,n_k) \\ 
	&= \int_{0}^{t} \ldots \int_{0}^{t} \sum_{\gamma^{x,y}} \prod_{i=1}^{k} \mathcal{P}^{M_2}_{n_i}(\gamma^{x,y}_{i-1},\gamma^{x,y}_{i}) \, d \overline{H}(n_1,\ldots,n_k) \\
	&\geq \int_{0}^{t} \ldots \int_{0}^{t} \sum_{\gamma^{x,y}} \prod_{i=1}^{k} \mathcal{P}^{M_1}_{n_i}(\gamma^{x,y}_{i-1},\gamma^{x,y}_{i}) \, d \overline{H}(n_1,\ldots,n_k) \\
	&= \int_{0}^{t} \ldots \int_{0}^{t} \mathcal{P}^{M_1}_{n_1} \mathcal{P}^{M_1}_{n_2} \ldots \mathcal{P}^{M_1}_{n_k}(x,y) \, d \overline{H}(n_1,\ldots,n_k) \\
	&= \mP_x\left(X^{M_1}(t) = y | N^{M_1}(t) = k\right).
	\end{align*}
\end{proof}

Another drawback of $X^{M_2}$ is that the speed-up in the cooling schedule depends on the difference between the two hill-climbing constants $c_{M_1}$ and $c_{M_2}$. One can expect that, when the state space $\mathcal{X}$ is large, Figure \ref{fig:a} is the generic case with $c_{M_1}$ and $c_{M_2}$ being close by. The advantage of their difference may perhaps be washed away by the extra computation  cost required by $X^{M_2}$.

\section{Conclusion and future work}\label{sec:conclusion}

In this paper we study a theoretically promising and improved variant of simulated annealing that we call $X^{M_2}$. This non-homogeneous continuous-time Markov chain enjoys a few desirable properties compared with its classical counterpart $X^{M_1}$: under general cooling schedule it escapes from the local minimum in the long run. As for strong ergodicity, we provide convergence guarantee in the two cases $c_{M_2} > 0$ and $c_{M_2} \leq 0$. In the former case one can adapt a fast logarithmic cooling, while in the latter case faster than logarithmic cooling are available and the algorithm still converges to the set of global minima. In additional to these long-run guarantees, we also give finite-time convergence estimate for $X^{M_2}$.

As mentioned in the Epilogue in Section \ref{sec:epilogue}, although there are a few limitations of $X^{M_2}$, we believe this work opens the door for studying new types or variants of classical annealing algorithms. Future work includes empirical investigation of $X^{M_2}$, extending $X^{M_2}$ to time-dependent target function \cite{Lowe96}, incorporating non-reversibility to speed up mixing \cite{Bie16,CH13,L97,Hwang05, Hwang93}, investigating the analog of $X^{M_2}$ in generalized simulated annealing by designing the so-called cost function to increase the spectral gap \cite{DelMiclo99}, extending the work to more general state space and its connection with the Langevin equation \cite{GW86,CHS87,ABD01}, investigating the analogue of $X^{M_2}$ in piecewise deterministic simulated annealing \cite{Mon16} and analyzing the cycle decomposition and energy landscape of $X^{M_2}$ \cite{Trouve96}.

\subsection*{Acknowledgements}
The author would like to thank Lu-Jing Huang, Chii-Ruey Hwang and Laurent Miclo for inspiring discussions that lead to this work. He would also like to thank Zhipeng Huang and three annonyomous referees for a careful reading. In particular, the author is grateful for various valuable suggestions and comments that the referees have made. The author acknowledges the support from the Chinese University of Hong Kong, Shenzhen grant PF01001143 and the financial support from AIRS - Shenzhen Institute of Artificial Intelligence and Robotics for Society Project 2019-INT002.

%
%
%
%
%

\bibliographystyle{abbrvnat}
\bibliography{thesis}

\begin{thebibliography}{37}
\providecommand{\natexlab}[1]{#1}
\providecommand{\url}[1]{\texttt{#1}}
\expandafter\ifx\csname urlstyle\endcsname\relax
  \providecommand{\doi}[1]{doi: #1}\else
  \providecommand{\doi}{doi: \begingroup \urlstyle{rm}\Url}\fi

\bibitem[Andrieu et~al.(2001)Andrieu, Breyer, and Doucet]{ABD01}
C.~Andrieu, L.~A. Breyer, and A.~Doucet.
\newblock Convergence of simulated annealing using {F}oster-{L}yapunov
  criteria.
\newblock \emph{J. Appl. Probab.}, 38\penalty0 (4):\penalty0 975--994, 2001.

\bibitem[Bertsimas and Tsitsiklis(1993)]{BT93}
D.~Bertsimas and J.~Tsitsiklis.
\newblock Simulated annealing.
\newblock \emph{Statist. Sci.}, 8\penalty0 (1):\penalty0 10--15, 02 1993.

\bibitem[Bierkens(2016)]{Bie16}
J.~Bierkens.
\newblock Non-reversible {M}etropolis-{H}astings.
\newblock \emph{Stat. Comput.}, 26\penalty0 (6):\penalty0 1213--1228, 2016.

\bibitem[Br\'{e}maud(1999)]{Bremaud99}
P.~Br\'{e}maud.
\newblock \emph{Markov chains}, volume~31 of \emph{Texts in Applied
  Mathematics}.
\newblock Springer-Verlag, New York, 1999.
\newblock Gibbs fields, Monte Carlo simulation, and queues.

\bibitem[Catoni(2004)]{C04}
O.~Catoni.
\newblock \emph{Statistical learning theory and stochastic optimization},
  volume 1851 of \emph{Lecture Notes in Mathematics}.
\newblock Springer-Verlag, Berlin, 2004.
\newblock Lecture notes from the 31st Summer School on Probability Theory held
  in Saint-Flour, July 8--25, 2001.

\bibitem[Chen and Hwang(2013)]{CH13}
T.-L. Chen and C.-R. Hwang.
\newblock Accelerating reversible {M}arkov chains.
\newblock \emph{Statist. Probab. Lett.}, 83\penalty0 (9):\penalty0 1956--1962,
  2013.

\bibitem[Chiang and Chow(1988)]{CC88}
T.-S. Chiang and Y.~S. Chow.
\newblock On eigenvalues and annealing rates.
\newblock \emph{Math. Oper. Res.}, 13\penalty0 (3):\penalty0 508--511, 1988.

\bibitem[Chiang et~al.(1987)Chiang, Hwang, and Sheu]{CHS87}
T.-S. Chiang, C.-R. Hwang, and S.~J. Sheu.
\newblock Diffusion for global optimization in {${\bf R}^n$}.
\newblock \emph{SIAM J. Control Optim.}, 25\penalty0 (3):\penalty0 737--753,
  1987.

\bibitem[Choi(2020{\natexlab{a}})]{Choi16}
M.~C. Choi.
\newblock Metropolis-{H}astings reversiblizations of non-reversible {M}arkov
  chains.
\newblock \emph{Stochastic Process. Appl.}, 130\penalty0 (2):\penalty0
  1041--1073, 2020{\natexlab{a}}.

\bibitem[Choi(2020{\natexlab{b}})]{Choi20}
M.~C. Choi.
\newblock
  \href{https://www.researchgate.net/publication/338478472_Hitting_mixing_and_tunneling_asymptotics_of_Metropolis-Hastings_reversiblizations_in_the_low-temperature_regime}{Hitting,
  mixing and tunneling asymptotics of Metropolis-Hastings reversiblizations in
  the low-temperature regime}.
\newblock \emph{Submitted}, 2020{\natexlab{b}}.

\bibitem[Choi and Huang(2020)]{CH18}
M.~C. Choi and L.-J. Huang.
\newblock On hitting time, mixing time and geometric interpretations of
  {M}etropolis-{H}astings reversiblizations.
\newblock \emph{J. Theoret. Probab.}, 33\penalty0 (2):\penalty0 1144--1163,
  2020.

\bibitem[Del~Moral and Miclo(1999)]{DelMiclo99}
P.~Del~Moral and L.~Miclo.
\newblock On the convergence and applications of generalized simulated
  annealing.
\newblock \emph{SIAM J. Control Optim.}, 37\penalty0 (4):\penalty0 1222--1250,
  1999.

\bibitem[Deuschel and Mazza(1994)]{DM94}
J.-D. Deuschel and C.~Mazza.
\newblock {$L^2$} convergence of time nonhomogeneous {M}arkov processes. {I}.
  {S}pectral estimates.
\newblock \emph{Ann. Appl. Probab.}, 4\penalty0 (4):\penalty0 1012--1056, 1994.

\bibitem[Fang et~al.(1997)Fang, Qian, and Gong]{FQG97}
H.~Fang, M.~Qian, and G.~Gong.
\newblock An improved annealing method and its large-time behavior.
\newblock \emph{Stochastic Process. Appl.}, 71\penalty0 (1):\penalty0 55--74,
  1997.

\bibitem[Freidlin and Wentzell(2012)]{FW12}
M.~I. Freidlin and A.~D. Wentzell.
\newblock \emph{Random perturbations of dynamical systems}, volume 260 of
  \emph{Grundlehren der Mathematischen Wissenschaften [Fundamental Principles
  of Mathematical Sciences]}.
\newblock Springer, Heidelberg, third edition, 2012.
\newblock Translated from the 1979 Russian original by Joseph Sz\"{u}cs.

\bibitem[Geman and Geman(1984)]{GG84}
S.~Geman and D.~Geman.
\newblock Stochastic relaxation, {G}ibbs distributions, and the {B}ayesian
  restoration of images.
\newblock \emph{IEEE Transactions on Pattern Analysis and Machine
  Intelligence}, PAMI-6\penalty0 (6):\penalty0 721--741, Nov 1984.

\bibitem[Geman and Hwang(1986)]{GW86}
S.~Geman and C.-R. Hwang.
\newblock Diffusions for global optimization.
\newblock \emph{SIAM J. Control Optim.}, 24\penalty0 (5):\penalty0 1031--1043,
  1986.

\bibitem[Gidas(1985)]{Gidas85}
B.~Gidas.
\newblock Global optimization via the {L}angevin equation.
\newblock In \emph{1985 24th IEEE Conference on Decision and Control}, pages
  774--778, Dec 1985.

\bibitem[H\"{a}ggstr\"{o}m(2002)]{H02}
O.~H\"{a}ggstr\"{o}m.
\newblock \emph{Finite {M}arkov chains and algorithmic applications}, volume~52
  of \emph{London Mathematical Society Student Texts}.
\newblock Cambridge University Press, Cambridge, 2002.

\bibitem[Holley and Stroock(1988)]{HS88}
R.~Holley and D.~Stroock.
\newblock Simulated annealing via {S}obolev inequalities.
\newblock \emph{Comm. Math. Phys.}, 115\penalty0 (4):\penalty0 553--569, 1988.

\bibitem[Hwang et~al.(1993)Hwang, Hwang-Ma, and Sheu]{Hwang93}
C.-R. Hwang, S.-Y. Hwang-Ma, and S.~J. Sheu.
\newblock Accelerating {G}aussian diffusions.
\newblock \emph{Ann. Appl. Probab.}, 3\penalty0 (3):\penalty0 897--913, 1993.

\bibitem[Hwang et~al.(2005)Hwang, Hwang-Ma, and Sheu]{Hwang05}
C.-R. Hwang, S.-Y. Hwang-Ma, and S.-J. Sheu.
\newblock Accelerating diffusions.
\newblock \emph{Ann. Appl. Probab.}, 15\penalty0 (2):\penalty0 1433--1444,
  2005.

\bibitem[Ingrassia(1994)]{I94}
S.~Ingrassia.
\newblock On the rate of convergence of the {M}etropolis algorithm and {G}ibbs
  sampler by geometric bounds.
\newblock \emph{Ann. Appl. Probab.}, 4\penalty0 (2):\penalty0 347--389, 1994.

\bibitem[Johnson and Isaacson(1988)]{JI88}
J.~Johnson and D.~Isaacson.
\newblock Conditions for strong ergodicity using intensity matrices.
\newblock \emph{J. Appl. Probab.}, 25\penalty0 (1):\penalty0 34--42, 1988.

\bibitem[Kirkpatrick(1984)]{K84}
S.~Kirkpatrick.
\newblock Optimization by simulated annealing: quantitative studies.
\newblock \emph{J. Statist. Phys.}, 34\penalty0 (5-6):\penalty0 975--986, 1984.

\bibitem[Kirkpatrick et~al.(1983)Kirkpatrick, Gelatt, and Vecchi]{KGv83}
S.~Kirkpatrick, C.~D. Gelatt, and M.~P. Vecchi.
\newblock Optimization by simulated annealing.
\newblock \emph{Science}, 220\penalty0 (4598):\penalty0 671--680, 1983.

\bibitem[Leisen and Mira(2008)]{LM08}
F.~Leisen and A.~Mira.
\newblock An extension of {P}eskun and {T}ierney orderings to continuous time
  {M}arkov chains.
\newblock \emph{Statist. Sinica}, 18\penalty0 (4):\penalty0 1641--1651, 2008.

\bibitem[L\"{o}we(1996)]{Lowe96}
M.~L\"{o}we.
\newblock Simulated annealing with time-dependent energy function via {S}obolev
  inequalities.
\newblock \emph{Stochastic Process. Appl.}, 63\penalty0 (2):\penalty0 221--233,
  1996.

\bibitem[L\"{o}we(1997)]{L97}
M.~L\"{o}we.
\newblock On the invariant measure of non-reversible simulated annealing.
\newblock \emph{Statist. Probab. Lett.}, 36\penalty0 (2):\penalty0 189--193,
  1997.

\bibitem[Miclo(1992)]{Miclo92}
L.~Miclo.
\newblock Recuit simul\'{e} sans potentiel sur un ensemble fini.
\newblock In \emph{S\'{e}minaire de {P}robabilit\'{e}s, {XXVI}}, volume 1526 of
  \emph{Lecture Notes in Math.}, pages 47--60. Springer, Berlin, 1992.

\bibitem[Miclo(1995)]{Miclo95}
L.~Miclo.
\newblock Une \'{e}tude des algorithmes de recuit simul\'{e} sous-admissibles.
\newblock \emph{Ann. Fac. Sci. Toulouse Math. (6)}, 4\penalty0 (4):\penalty0
  819--877, 1995.

\bibitem[Miclo(1996)]{Miclo96}
L.~Miclo.
\newblock Sur les probl\`emes de sortie discrets inhomog\`enes.
\newblock \emph{Ann. Appl. Probab.}, 6\penalty0 (4):\penalty0 1112--1156, 1996.

\bibitem[Monmarch\'{e}(2016)]{Mon16}
P.~Monmarch\'{e}.
\newblock Piecewise deterministic simulated annealing.
\newblock \emph{ALEA Lat. Am. J. Probab. Math. Stat.}, 13\penalty0
  (1):\penalty0 357--398, 2016.

\bibitem[Peskun(1973)]{Pesk73}
P.~H. Peskun.
\newblock Optimum {M}onte-{C}arlo sampling using {M}arkov chains.
\newblock \emph{Biometrika}, 60:\penalty0 607--612, 1973.

\bibitem[Tierney(1998)]{Tie98}
L.~Tierney.
\newblock A note on {M}etropolis-{H}astings kernels for general state spaces.
\newblock \emph{Ann. Appl. Probab.}, 8\penalty0 (1):\penalty0 1--9, 1998.

\bibitem[Trouv\'{e}(1996)]{Trouve96}
A.~Trouv\'{e}.
\newblock Cycle decompositions and simulated annealing.
\newblock \emph{SIAM J. Control Optim.}, 34\penalty0 (3):\penalty0 966--986,
  1996.

\bibitem[van Dijk(1992)]{vDN92}
N.~M. van Dijk.
\newblock Uniformization for nonhomogeneous {M}arkov chains.
\newblock \emph{Oper. Res. Lett.}, 12\penalty0 (5):\penalty0 283--291, 1992.

\end{thebibliography}

\end{document}